\documentclass{amsart}
\usepackage{amsthm,amsmath}
\usepackage{amssymb, amsfonts}
\usepackage{hyperref}
\usepackage{dsfont}
\usepackage{mathtools}
\usepackage{upgreek}
\usepackage{fancyhdr}
\usepackage{tikz}
\usepackage{chngcntr}
\usepackage{wasysym}
\usepackage{tikz-cd}
\usepackage{color}
\usepackage{ulem}
\usepackage[vskip=0pt]{quoting}
\usetikzlibrary{trees}

\usepackage{chngcntr}

\newtheorem{thm}{Theorem}[section]
\newtheorem{prop}[thm]{Proposition}
\newtheorem{cor}[thm]{Corollary}
\newtheorem{lemma}[thm]{Lemma}
\newtheorem{fact}[thm]{Fact}

\theoremstyle{definition}
\newtheorem{rmk}[thm]{Remark}
\newtheorem{defn}[thm]{Definition}
\newtheorem{nota}[thm]{Notation}
\newtheorem{obs}[thm]{Observation}
\newtheorem{claim}[thm]{Claim}

\newtheorem*{claimx}{Claim}
\newtheorem*{claim1}{Claim 1}
\newtheorem*{claim2}{Claim 2}
\newtheorem*{claim3}{Claim 3}

\newtheorem*{subsubclaim}{Subsubclaim}
\newtheorem*{subclaim1}{Subclaim 1}
\newtheorem*{subclaim2}{Subclaim 2}

\counterwithin{figure}{thm}

\newtheorem*{thmx}{Theorem \ref{beta}}
\newtheorem*{thmy}{Theorem \ref{nu str}}
\newtheorem*{propx}{Proposition \ref{modifi}}
\newtheorem*{corx}{Corollary \ref{SOP1NSOP2->ATP}}

\newtheorem*{case1}{Case 1 [$\sigma_i\perp\sigma_j$]}

\def \tree {{^{\omega >}}2}
\def \trel {{{^{{\omega_1} >}}}2}
\def \tren {{{^{n>}}2}}
\def \trek {{{^{k>}}2}}
\def \tri {\trianglelefteq}
\def \trr {\trianglerighteq}
\newcommand{\trn}{%
  \mathrel{\ooalign{$\lneq$\cr\raise.21ex\hbox{$\lhd$}\cr}}}
\newcommand{\trrn}{%
  \mathrel{\ooalign{$\gneq$\cr\raise.21ex\hbox{$\rhd$}\cr}}}
\def \coc {{^{\frown}}}
\def \lr {{\langle \rangle}}
\def \lor {{\langle 0 \rangle}}
\def \llr {{\langle 1 \rangle}}
\def \lll {{\langle l \rangle}}
\def \res {\lceil}
\def \cl {{\rm cl}}
\def \paratr {\langle a_{\eta}\rangle_{\eta\in\tree}}

\def \la {\langle}
\def \ra {\rangle}
\def \L {\mathcal{L}}

\def \sopl {\operatorname{SOP}_1}

\def \sopz {\operatorname{SOP}_2}

\title{SOP$_1$, SOP$_2$, and antichain tree property}
\author[JinHoo Ahn and Joonhee Kim]{JinHoo Ahn and Joonhee Kim}
\date {\today}
\thanks{The first author has been supported by Samsung Science Technology Foundation under project number SSTF-BA1301-03 and a KIAS Individual Grant (CG079801) at Korea Institute for Advanced Study. The second author has been supported by Samsung Science Technology Foundation under project number SSTF-BA1301-03, NRF of Korea grants 2018R1D1A1A02085584 and 2021R1A2C1009639, and a KIAS Individual Grant (6G091801) at Korea Institute for Advanced Study.}

\begin{document}

\maketitle

\begin{abstract}
In this paper, we study some tree properties and their related indiscernibilities. First, we prove that SOP$_2$ can be witnessed by a formula with a tree of tuples holding `arbitrary homogeneous inconsistency' (e.g., weak $k$-TP$_1$ conditions or other possible inconsistency configurations). 

And we introduce a notion of tree-indiscernibility, which preserves witnesses of SOP$_1$, and by using this, we investigate the problem of (in)equality of SOP$_1$ and SOP$_2$. 

Assuming the existence of a formula having SOP$_1$ such that no finite conjunction of it has SOP$_2$, we observe that the formula must witness some tree-property-like phenomenon, which we will call the antichain tree property (ATP, see Definition \ref{def of AT}). We show that ATP implies SOP$_1$ and TP$_2$, but the converse of each implication does not hold. So the class of NATP theories (theories without ATP) contains the class of NSOP$_1$ theories and the class of NTP$_2$ theories.

At the end of the paper, we construct a structure whose theory has a formula having ATP, but any conjunction of the formula does not have SOP$_2$. So this example shows that  SOP$_1$ and SOP$_2$ are not the same at the level of formulas, i.e., there is a formula having SOP$_1$, while any finite conjunction of it does not witness SOP$_2$ (but a variation of the formula still has SOP$_2$).
\end{abstract}

\section{Introduction}
Many model-theoretic properties can be presented by using a tree configuration. For example, a theory $T$ is simple if and only if $T$ does not have any formula having $k$-tree property for some $k>1$. That is, there are no formula $\varphi(x,y)$ and tuple $(c_\eta:\eta\in\null^{\omega>}\omega)$ such that $\{\varphi(x,c_{\eta^\frown n}):n<\omega\}$ is $k$-inconsistent for each $\eta\in\null^{\omega>}\omega$ and $\{\varphi(x,c_{\nu\lceil i}):i<\omega\}$ is consistent for each $\nu\in\null^{\omega}\omega$ [\ref{Kim}, Proposition 2.3.7].

The notion of tree property was extended to SOP$_1$ and SOP$_2$ by D{\v z}amonja and Shelah [\ref{DS}], and these properties are further studied by Shelah and Usvyatsov [\ref{SU}]. It is known that the implication
\begin{center}
    SOP$_2$ $\Rightarrow$ SOP$_1$ $\Rightarrow$ TP(non-simple)
\end{center}
holds for $T$ and even more, the second implication is proper \cite{DS}. But unfortunately, it is still an open question whether the first one is reversible. In this paper, we focus on this problem and point out that a theory $T$ must have some properties if it has a formula having SOP$_1$, but any conjunction of it does not have SOP$_2$.\par
We introduce one of the attempts to show whether SOP$_1$ and SOP$_2$ are equivalent or not. The following claim is from \cite{DS}, which gave the authors a strong motivation.
     \begin{claim}\label{claim}\cite[Claim 2.15]{DS}
     Suppose that $\varphi(x,y)$ satisfies SOP$_1$, but there is no $n\in\omega$ such that the formula $\varphi_n(x,y_0,...,y_{n-1})=\bigwedge_{k<n}\varphi(x,y_k)$ satisfies SOP$_2$. Then there are witnesses $\la a_\eta\ra_{\eta\in\tree}$ for $\varphi(x,y)$ satisfying SOP$_1$ which in addition satisfy\par
     \begin{enumerate}
         \item [(i)] if $X\subseteq \null^{\omega>}2$, and there are no $\eta, \nu\in X$ such that $\eta^\frown\langle0\rangle\unlhd\nu$, then $\{\varphi(x,a_\eta):\eta\in X\}$ is consistent,
         \item [(ii)] $\la a_\eta\ra_{\eta\in\tree}$ is 1-fbti.
     \end{enumerate}
     \end{claim}
The claim is not true, even if the last condition is deleted. For example, if we choose $X=\{\la0\ra, \la 1\ra\}$, then $X$ has no $\eta, \nu$ such that $\eta\coc\la0\ra\tri\nu$ but $\{\varphi(x,a_\eta):\eta\in X\}$ can not be consistent since $\la a_\eta\ra_{\eta\in\tree}$ witnesses SOP$_1$ with $\varphi(x,y)$. So there is no witness of SOP$_1$ satisfying (i).\footnote{After \cite{DS} was published, the authors of \cite{DS} released a modified version of \cite{DS} online. In \cite{DS2}, the new version of \cite{DS}, they introduce $3$-fbti, a notion of indiscernibility weaker than 1-fbti and rewrite \cite[Claim 2.15]{DS} using this as follows:
\begingroup
\addtolength\leftmargini{-0.18in}
\begin{quote}
\cite[Claim 2.19]{DS2}
     Suppose that $\varphi(x,y)$ satisfies SOP$_1$, but there is no $n\in\omega$ such that the formula $\varphi_n(x,y_0,...,y_{n-1})=\bigwedge_{k<n}\varphi(x,y_k)$ satisfies SOP$_2$. Then there are witnesses $\la a_\eta\ra_{\eta\in\tree}$ for $\varphi(x,y)$ satisfying SOP$_1$ which in addition satisfy\par
     \begin{enumerate}
         \item [(i)] if $X\subseteq \null^{\omega>}2$, and there are no $\eta, \nu\in X$ such that $\eta^\frown\langle0\rangle\unlhd\nu$, then $\{\varphi(x,a_\eta):\eta\in X\}$ is consistent,
         \item [(ii)] $\la a_\eta\ra_{\eta\in\tree}$ is 3-fbti.
     \end{enumerate}
\end{quote}
\endgroup
\noindent But the same problem still exists with this changed claim.}
However, we elaborate on the proof to investigate two regions in Shelah's classification program; theories having SOP$_2$ and theories having SOP$_1$ but not having SOP$_2$. \par
Here is the outline of this paper.

In Section 2, we recall {\it strongly indiscernible trees}, then observe that SOP$_2$, TP$_1$, and other related tree properties can be described uniformly by using the notions $A^{str}$-SOP$_2$ and $A^{str}$-TP$_1$(Definition \ref{A str}). With antichains in trees, these notions give us a more generalized version of SOP$_2$ and TP$_1$. More specifically, we show the following:
\begin{thmy}
 Let $1<k<\omega$ be given.
 \begin{enumerate} 
     \item For any antichain tuple $\bar{\nu}\subseteq\null^{\omega>}\omega$ of size $k$, $T$ has TP$_1$ if and only if $T$ has $\bar{\nu}^{str}$-TP$_1$.
     \item For any antichain tuple $\bar{\nu}\subseteq\null^{\omega>}2$ of size $k$, $T$ has SOP$_2$ if and only if $T$ has $\bar{\nu}^{str}$-SOP$_2$.
 \end{enumerate}
\end{thmy}

In Section 3, we recall the notions of $\alpha$-indiscernibility and $\beta$-indiscernibility\footnote{$\alpha$-indiscernibility and $\beta$-indiscernibility appear in \cite{DS2} under the name $1$-fbti and $3$-fbti, respectively. See \cite[Definition 2.10, 2.14]{DS2}.}, and show that $\beta$-indiscernibility preserves SOP$_1$. Namely,
\begin{thmx}
If $\varphi(x,y)$ witnesses SOP$_1$, then there exists a $\beta$-indiscernible tree $\la e_{\eta}\ra_{\eta\in\tree}$ which witnesses SOP$_1$ with $\varphi$.
\end{thmx}
\noindent We will apply it to further problems. The outline of the proof of Theorem \ref{beta} is taken from [\ref{DS2}], a revised version of \cite{DS}. However, the proof of \cite{DS2} is missing an important part, so it is difficult to say that the desired conclusion has been fully proved. We make up for the missing parts and include complete proof.

In Section 4, we introduce a notion of tree property, which is called {\it antichain tree property (ATP)}, and investigate the relationship between ATP and other tree properties. In particular, we show that ATP implies SOP$_1$ and TP$_2$, so we have the following diagram:
\begin{center}
\begin{tikzcd}
 & {\rm NIP}  \arrow{r} & {\rm NTP}_2\arrow{r} & {\rm NATP} \\ 
& {\rm stable} \arrow{r} \arrow{u} & {\rm simple}\arrow{r} \arrow{u} & {\rm NSOP}_1 \arrow{u} .
\end{tikzcd}
\end{center}

\noindent Another main result of Section 4 is below.

\begin{corx}
If there exists a formula having SOP$_1$ but any conjunction of it does not have SOP$_2$, then the theory has ATP. The witness of ATP can be selected to be strongly indiscernible.
\end{corx}
So if there exists an NSOP$_2$ theory having ATP, then $\sopl\supsetneq\sopz$. On the contrary, if the existence of an antichain tree always implies the existence of witness of SOP$_2$, then $\sopl=\sopz$.

 In Section 5, we introduce {\it 1 strong order property with full consistency} (we will write it SOP$^{\rm fc}_1$). By definition, if a theory has SOP$^{\rm fc}_1$, then it has SOP$_1$. We show that $\beta$-indiscernibility also preserves SOP$^{\rm fc}_1$ and prove that a formula has ATP if and only if it has SOP$^{\rm fc}_1$. Thus the class of theories having SOP$^{\rm fc}_1$ is a proper subclass of the class of SOP$_1$ theories. 
 
 As a modification of Claim \ref{claim}, we observe that the following proposition holds.

\begin{propx}
Suppose that $\varphi(x,y)$ satisfies SOP$_1$, but there is no $n\in\omega$ such that the formula $\varphi_n(x,y_0,...,y_{n-1})=\bigwedge_{k<n}\varphi(x,y_k)$ satisfies SOP$_2$. Then there are witnesses $\la a_\eta\ra_{\eta\in\tree}$ for $\varphi(x,y)$ satisfying SOP$_1$ which in addition satisfy
\begin{enumerate}
\item [(i)] $\la a_\eta\ra_{\eta\in\tree}$ witnesses SOP$^{\rm fc}_1$ with $\varphi$,
\item [(ii)] $\la a_\eta\ra_{\eta\in\tree}$ is $\beta$-indiscernible.
\end{enumerate}
\end{propx}

In Section 6, we construct an example of theory having ATP. Our example is a structure $\mathcal{C}$ in relational language $\mathcal{L}=\{R\}$ where $R$ is a binary relation symbol. We show that there exists a formula $\varphi(x,y)$ having ATP (so it has SOP$_1$), and $\bigwedge_{i<n}\varphi(x,y_i)$ does not witness SOP$_2$ modulo ${\rm Th}(\mathcal{C})$, for any $n\in\omega$.

\medskip

We use standard notations for trees.
For ordinals $\kappa$ and $\lambda$, we write $^{<\lambda}\kappa$ for the tree of height $\lambda$ and $\kappa$ many branches. In other words, $^{<\lambda}\kappa$ is ${\bigcup_{\lambda'<\lambda}}(^{\lambda'}\kappa)$ where ${^{\lambda'}}\kappa$ is the set of all functions from $\lambda'$ to $\kappa$. We denote $\unlhd$ to be a natural partial order in the tree, that is, $\eta \unlhd \nu$ if $\eta=\nu$ or $\nu \lceil \alpha = \eta$ for some ordinal $\alpha \in$ dom($\nu)$. Note that $\eta\tri\nu$ if and only if $\eta\subseteq\nu$.

We may use $\langle\rangle$ and $\emptyset$ as an empty string, $0^\alpha$ as a string of $\alpha$ many zeros, $1^\alpha$ as a string of $\alpha$ many ones, and $\alpha$ as a string $\langle\alpha\rangle$ of length one, because each element in tree can be considered as a string. $\emptyset$ may be interpreted as an empty set or an empty string depending on the context.

In addition, we give three more notations for trees.
First, we say $\xi = \eta \wedge \nu$ if $\xi$ is the meet of $\eta$ and $\nu$, i.e., $\xi=\eta \lceil \beta$, when $\beta = \bigcup \{ \alpha \leq $ dom$(\eta) \cap $dom$(\nu) : \eta \lceil \alpha=\nu \lceil \alpha \}$. For $\bar{\eta}\subseteq\null^{<\lambda}\kappa$, $\bar{\nu}$ is the meet closure of $\bar{\eta}$ if $\bar{\nu}=\{ \eta_1 \wedge \eta_2 : \eta_1, \eta_2 \in \bar{\eta} \}$. 
Second, we say $\eta<_{lex}\nu$ if $\eta \lhd \nu$, or $\eta$ and $\nu$ are $\unlhd$-incomparable($\eta \perp \nu$) and for ordinal $\alpha=$ dom$(\eta\wedge\nu),\,  \eta(\alpha)<\nu(\alpha)$. Finally, we denote len$(\eta)$ and $l(\eta)$ to be the length of $\eta$.

\section{Strongly indiscernible trees and SOP$_2$}\label{SOP2}

First we recall the definitions and facts for trees and tree indiscernibilities in \cite{DS}, \cite{KK}, \cite{KKS}, \cite{Sco15}, and \cite{TT}.\par

\begin{defn}
The {\it strong language} $\mathcal{L}_0$ is defined to be the collection $\{ \lhd, \wedge, <_{lex} \}$, where $\lhd$, $<_{lex}$ are binary relation symbols, $\wedge$ is a binary function symbol.
  \end{defn}

  \begin{defn}
Let $\kappa$ and $\lambda$ be ordinals. Then we can give an $\L_0$-structure on $\null^{<\lambda}\kappa$ by interpreting each symbol as in the last paragraph of Section 1.  Let $\L$ be a language, $T$ a complete $\L$-theory, and $\mathcal{M}$ a monster models of $T$. For a tree indexed set $\la b_\eta : \eta \in \null^{<\lambda}\kappa\ra$ in $\mathcal{M}$, we say it is {\it strongly indiscernible} if for any finite tuples $\bar{\eta}$ and $\bar{\nu}$ with the same quantifier-free $\mathcal{L}_0$-type in $^{<\lambda}\kappa$, $\la b_\eta\ra_{\eta\in \bar{\eta}} \equiv \la b_\nu\ra_{\nu\in \bar{\nu}}$. Generally, for a given structure $\mathcal{I}$ with language $L_{\mathcal{I}}$, we say a set $\{b_i : i \in \mathcal{I}\}$ is {\it $\mathcal{I}$-indexed indiscernible} if $\la b_i\ra_{i\in \bar{i}} \equiv \la b_j\ra_{j\in \bar{j}}$ for any finite $\bar{i}$, $\bar{j}\subseteq\mathcal{I}$ with the same quantifier-free $\L_\mathcal{I}$-type.
\end{defn}
 We say $\bar{\eta}$ is strongly similar to $\bar{\nu}$ ($\bar{\eta}\sim_{str}\bar{\nu}$), if they have the same quantifier-free type in $\L_0$-structure.

\begin{defn}
 Let $\L$ and $\L'$ be languages, $T$ an $\L$-theory, $\mathcal{M}$ a monster model of $T$, and $\mathcal{I}$ an $\L'$-structure. We say $B=\{ b_\eta\in\mathcal{M} : \eta \in \mathcal{I} \}$ is {\it locally based on} a set $A=\{ a_\eta\in\mathcal{M} : \eta \in \mathcal{I} \}$  if for all $\varphi(x_{i_1}, \dots , x_{i_n})$ in $\mathcal{L}$ and for all $\eta_1, \dots , \eta_n \in \mathcal{I}$, there are some $\nu_1, \cdots , \nu_n \in \mathcal{I}$ such that \par
  \begin{itemize}
     \item[(a)] $\nu_1 \dots \nu_n$ and $\eta_1 \dots \eta_n$ have the same quantifier-free $\L'$-type in $\mathcal{I}$, and
     \item[(b)] $b_{\eta_1} \dots b_{\eta_n}\equiv_\varphi a_{\nu_1} \dots a_{\nu_n}$
  \end{itemize}
    In particular, when $\mathcal{I}$ is $\L_0$-structure $\null^{<\lambda}\kappa$, we say $B$ is {\it strongly locally based on} $A$ whenever $B$ is locally based on $A$.
\end{defn}

\begin{defn}
For an index structure $\mathcal{I}$, we say $\mathcal{I}$-indexed indiscernibles have the {\it modeling property} if for any set of parameters $A=\{ a_\eta : \eta \in \mathcal{I} \}$, there is an $\mathcal{I}$-indexed indiscernible $B=\{ b_\eta : \eta \in \mathcal{I} \}$ such that $B$ is locally based on $A$.
    \end{defn}

\begin{fact}\label{modeling} \cite{TT}\cite{Sco15}
Let $^{<\omega}\omega$ be the universe of the index structure. Then the strong indiscernibles have the modeling property.
\end{fact}
 The fact says that given a set $A=\{ a_\eta : \eta \in \!\null^{<\omega}\omega \}$, we always have $B=\{ b_\eta : \eta \in \null^{<\omega}\omega \}$ which is strongly indiscernible and strongly locally based on $A$.\par

\bigskip
Now we recall the definition of 2-strong order property and its related notions from \cite{KK}.
\begin{defn}
 Let $\varphi(x,y)$ be a formula in $\mathcal{L}$. Let $1<k<\omega$ be given
 \begin{enumerate}
     \item We say $\varphi(x,y)$ has {\it 2-strong order property} (SOP$_2$) if there is a tree $\la a_{\eta}\ra_{\eta \in \null^{<\omega}2}$ such that
       \begin{enumerate}
        \item[(1)] for all $\eta \in \null^{\omega}2$, $\{ \varphi (x,a_{\eta \lceil \alpha}) : \alpha < \omega \}$ is consistent, and
        \item[(2)] for all $\eta \perp \nu \in \null^{<\omega}2$, $\{ \varphi (x,a_{\eta}), \varphi (x,a_{\nu}) \}$ is inconsistent.
       \end{enumerate}
     \item We say $\varphi(x,y)$ has {\it tree property of the first kind} (TP$_1$) if there is a tree $\la a_{\eta}\ra_{\eta \in \null^{<\omega}\omega}$ such that
       \begin{enumerate}
        \item[(1)] for all $\eta \in \null^{\omega}\omega$, $\{ \varphi (x,a_{\eta \lceil \alpha}) : \alpha < \omega \}$ is consistent, and
        \item[(2)] for all $\eta \perp \nu \in \null^{<\omega}\omega$, $\{ \varphi (x,a_{\eta}), \varphi (x,a_{\nu}) \}$ is inconsistent.
       \end{enumerate}
      \item We say $\varphi(x,y)$ has $k$-TP$_1$ if there is a tree $\la a_{\eta}\ra_{\eta \in \null^{<\omega}\omega}$ such that
       \begin{enumerate}
        \item[(1)] for all $\eta \in \null^{\omega}\omega$, $\{ \varphi (x,a_{\eta \lceil \alpha}) : \alpha < \omega \}$ is consistent, and
        \item[(2)] for any pairwise incomparable $\eta_0, \ldots, \eta_{k-1} \in \null^{<\omega}\omega$, $\{ \varphi (x,a_{\eta_i}):i<k \}$ is inconsistent.
       \end{enumerate}
      \item We say $\varphi(x,y)$ has weak $k$-TP$_1$ if there is a tree $\la a_{\eta}\ra_{\eta \in \null^{<\omega}\omega}$ such that
       \begin{enumerate}
        \item[(1)] for all $\eta \in \null^{\omega}\omega$, $\{ \varphi (x,a_{\eta \lceil \alpha}) :\alpha < \omega \}$ is consistent, and
        \item[(2)] for any $\eta, \eta_0, \ldots, \eta_{k-1} \in \null^{<\omega}\omega$ and $i_0<\cdots<i_{k-1}<\omega$, if $\eta^\frown\la i_l\ra\unlhd\eta_i$ for each $l<k$, then $\{ \varphi (x,a_{\eta_i}):i<k \}$ is inconsistent.
       \end{enumerate}
 \end{enumerate}
\end{defn}
We say $T$ has SOP$_2$ if it has a SOP$_2$ formula. Otherwise, we say $T$ is NSOP$_2$. Similarly, we say $T$ has TP$_1$, $k$-TP$_1$, or weak $k$-TP$_1$ if it has a formula with that property. If $T$ has weak $k$-TP$_1$ for some $k$, then we say $T$ has {\it weak TP$_1$.} Note that if $T$ has TP$_1$, then it has $k-$TP$_1$ for some $k$ by definitions. It is also clear that if $T$ has $k-$TP$_1$ for some $k$, then it has weak $k-$TP$_1$. In fact, these three notions are equivalent by the following observation of Chernikov and Ramsey \cite{CR}.

\begin{fact}\cite[Theorem 4.8]{CR}\label{fact:weak TP1 TP1}
Given $k\ge 2$, $T$ has weak $k-$TP$_1$ if and only if $T$ has TP$_1$.
\end{fact}

 Note that in the definition of TP$_1$, $k$-TP$_1$, or weak $k$-TP$_1$, we may assume the tree $\la a_{\eta}\ra_{\eta \in \null^{<\omega}\omega}$ is strongly indiscernible by Fact \ref{modeling}. On the other hand, since SOP$_2$ deals with the binary tree $^{<\omega}2$ as an index set, we can not directly apply \ref{modeling} into it. The following fact, however, says that it is possible for SOP$_2$ to assume $\la a_{\eta}\ra_{\eta \in \null^{<\omega}2}$ is strongly indiscernible, too.
\begin{lemma}\label{str indisc tree on SOP2} \cite[Lemma 4.3 (3)]{CR}
For any formula $\varphi(x,y)$ having SOP$_2$, there is a strongly indiscernible tree $\la a_{\eta}\ra_{\eta \in \null^{<\omega}2}$ witnessing the same property.
\end{lemma}

If we apply the concept of strong indiscernibility and antichains, these tree properties can be classified in a more uniform way.
 
\begin{defn}
A subset $A\subseteq\null^{\kappa>}\lambda$ is called an antichain if for all $\eta, \nu\in A$, $\eta\perp\nu$.
\end{defn}
Similarly, we say a tuple $\bar{\eta}$ is an antichain if the set $\{\eta:\eta\in\bar{\eta}\}$ is an antichain.

\begin{defn} \label{A str}
 \begin{enumerate}
     \item Let $A$ be a set of tuples in $\null^{\omega>}2$. We say $\varphi$ has $A^{str}$-SOP$_2$ if there is $\la a_\eta:\eta\in\null^{\omega>}2\ra$ such that\par
      \begin{enumerate}
         \item[(a)] for all $\eta \in \null^\omega2$, $\{\varphi(x,a_{\eta\lceil m}):m<\omega\}$ is consistent, and
         \item[(b)] for all $\bar{\nu}\subseteq\null^{\omega>}2$, if $\bar{\nu}\sim_{str}\bar{\xi}$ for some $\bar{\xi}\subseteq A$, then $\{\varphi(x,a_{\nu}):\nu\in\bar{\nu}\}$ is inconsistent.
      \end{enumerate}
     \item Let $A$ be a set of tuples in $\null^{\omega>}\omega$. We say $\varphi$ has $A^{str}$-TP$_1$ if there is $(a_\eta:\eta\in\null^{\omega>}\omega)$ such that\par
      \begin{enumerate}
         \item[(a)] for all $\eta \in \null^\omega\omega$, $\{\varphi(x,a_{\eta\lceil m}):m<\omega\}$ is consistent, and
         \item[(b)] for all $\bar{\nu} \subseteq\null^{\omega>}\omega$, if $\bar{\nu}\sim_{str}\bar{\xi}$ for some $\bar{\xi}\subseteq A$, then $\{\varphi(x,a_{\nu}):\nu\in\bar{\nu}\}$ is inconsistent.
      \end{enumerate}
     \item We say $T$ has $A^{str}$-SOP$_2$ (resp. $A^{str}$-TP$_1$) if it has a $A^{str}$-SOP$_2$ (resp. $A^{str}$-TP$_1$) formula. If $A=\{\bar{\nu}\}$, then we say $\varphi$ (or $T$) has $\bar{\nu}^{str}$-SOP$_2$ (resp. $\bar{\nu}^{str}$-TP$_1$). 
 \end{enumerate}
\end{defn}

\begin{rmk}
 \begin{enumerate}
     \item $\varphi$ has SOP$_2$ if and only if $\varphi$ has $\langle\langle0\rangle,\langle1\rangle\rangle^{str}$-SOP$_2$.
     \item $\varphi$ has TP$_1$ if and only if $\varphi$ has $\langle\langle0\rangle,\langle1\rangle\rangle^{str}$-TP$_1$.
     \item $\varphi$ has $k$-TP$_1$ if and only if for the collection $A$ of all antichain tuples of size $k$, $\varphi$ has $A^{str}$-TP$_1$.
     \item $\varphi$ has weak $k$-TP$_1$ if and only if $\varphi$ has $\langle\langle0\rangle, \ldots, \langle k-1\rangle\rangle^{str}$-TP$_1$.
 \end{enumerate}
\end{rmk}

We aim to prove the following;
\begin{thm} \label{nu str}
 Let $1<k<\omega$ be given.
 \begin{enumerate} 
     \item For any antichain tuple $\bar{\nu}\subseteq\null^{\omega>}\omega$ of size $k$, $T$ has TP$_1$ if and only if $T$ has $\bar{\nu}^{str}$-TP$_1$.
     \item For any antichain tuple $\bar{\nu}\subseteq\null^{\omega>}2$ of size $k$, $T$ has SOP$_2$ if and only if $T$ has $\bar{\nu}^{str}$-SOP$_2$.
 \end{enumerate}
\end{thm}
 
To do this, we first observe lemmas of monochromatic subtrees in the colored tree.
\begin{lemma}\label{monochrom lemma} \cite[Lemma 2.16]{DS}
 Suppose $\kappa$ is a regular cardinal and we color $^{\kappa>}2$ by $\theta<\kappa$ colors. Let $c$ be the given coloring.\par
 (1) There is $\nu^{\ast}$ in $^{\kappa>}2$ and $j<\theta$ such that for any $\nu\unrhd\nu^{\ast}$ we can find $\rho\unrhd\nu$ the color of which is $j$.\par
 (2) There is an embedding $h: \null^{\omega>}2\rightarrow \null^{\kappa>}2$ such that
 \begin{itemize}
     \item $h(\eta)^\frown\langle0\rangle\unlhd h(\eta^\frown\langle0\rangle)$, and $h(\eta)^\frown\langle1\rangle\unlhd h(\eta^\frown\langle1\rangle)$
     \item Ran$(h)$ is monochromatic, that is, for any $\eta$ and $\nu$ in $\null^{\omega>}2$, the color of $h(\eta)$ is as same as that of $h(\nu)$.
 \end{itemize}
\end{lemma}

This can be easily modified into the lemma on $\omega$-branched trees.
\begin{lemma}\label{monochrom lemma 2}
 Suppose $\kappa$ is a regular cardinal and we color $^{\kappa>}\omega$ by $\theta<\kappa$ colors. Let $c$ be the given coloring.\par
 (1) There is $\nu^{\ast}$ in $^{\kappa>}\omega$ and $j<\theta$ such that for any $\nu\unrhd\nu^{\ast}$ we can find $\rho\unrhd\nu$ the color of which is $j$.\par
 (2) There is an embedding $h: \null^{\omega>}\omega\rightarrow \null^{\kappa>}2$ such that
 \begin{itemize}
     \item $h(\eta)^\frown\langle i\rangle\unlhd h(\eta^\frown\langle i\rangle)$ for each $i<\omega$
     \item Ran$(h)$ is monochromatic, that is, for any $\eta$ and $\nu$ in $\null^{\omega>}\omega$, the color of $h(\eta)$ is as same as that of $h(\nu)$.
 \end{itemize}
\end{lemma}
\begin{proof}
(1) Suppose not. For $i<\theta$, we inductively choose $\eta_i\in\null^{\kappa>}\omega$ such that 
\begin{itemize}
    \item if $i<j$, then $\eta_i\unlhd\eta_j$, and
    \item for $\rho\in\null^{\kappa>}\omega$, if $\eta_{i+1}\unlhd\rho$, then $c(\rho)\neq i$.
\end{itemize}
Let $\nu=\bigcup_{i<\theta}\eta_i$. Since $\theta<$ cf$(\kappa)$, len$(\nu)$ cannot be $\kappa$, so $\nu\in\null^{\kappa>}\omega$. But this contradicts that $c(\nu)$ has no color. \par
(2) Fix $\nu^{\ast}$ and $j<\theta$ in (1). Find $\xi\unrhd\nu^{\ast}$ such that $c(\xi)=j$ then choose $h(\la\ra)$ to be $\xi$. In the same manner, assuming $h(\eta)$ is already chosen to be some $\nu\in\null^{\kappa>}\omega$, we find $\xi_i\unrhd \nu^\frown i$ for each $i<\omega$ such that $c(\xi_i)=j$, then choose $h(\eta^\frown\la i\ra)$ to be $\xi_i$. This inductive process gives us the desired map.
\end{proof}



Using the two preceding lemmas, we find equivalent conditions of SOP$_2$ and weak TP$_1$ respectively, as the index structures that they rely on are different. 

 \begin{prop}\label{char nsop2}
 For each $\eta\in\null^{\omega_1>}2$, $m<\omega$, $\alpha\leq\omega_1$, let $K_{\eta, m, \alpha}$ be the set $\{\eta^\frown\nu^\frown0^\beta:\nu\in\null^m2, \beta<\alpha\}$, and $O_\eta$ be the set $\{\eta^\frown0^\beta:\beta<\omega_1\}$.\par
 The following are equivalent.
 \begin{enumerate}
     \item $T$ is NSOP$_2$.
     \item For all $\varphi(x,y)$ and strongly indiscernible tree $\la a_\eta:\eta\in\null^{\omega_1>}2\ra$, if $\{\varphi(x,a_\nu):\nu\in O_{\langle\rangle}\}$ is consistent, then for each $\eta\in\null^{\omega_1>}2$ and $m<\omega$, $\{\varphi(x,a_\nu):\nu\in K_{\eta, m, \omega_1}\}$ is consistent.
 \end{enumerate}
 \end{prop}

  \begin{proof}
  ($1\Leftarrow2$) Use \ref{str indisc tree on SOP2} and compactness. \par
  ($1\Rightarrow2$) Suppose not. Fix $\varphi$ and $\la a_\eta:\eta\in\null^{\omega_1>}2\ra$. We inductively choose a finite subset $w_\eta\subseteq\null^{\omega_1>}2$ and $\nu_\eta\in\null^{\omega_1>}2$ for each $\eta\in\null^{\omega_1>}2$ so that the following conditions hold after the construction;\par
  \begin{enumerate}
      \item[(a)] the union of $\{\varphi(x,a_\nu):\nu\in\bigcup\{w_{\eta\lceil\alpha}:\alpha\leq\text{len}(\eta)\}\}$ and $\{\varphi(x,a_\nu):\nu\in O_{\nu_\eta}\}$ is consistent, and
      \item[(b)] the union of $\{\varphi(x,a_\nu):\nu\in\bigcup\{w_{\eta\lceil\alpha}:\alpha\leq\text{len}(\eta)\}\}$ and $\{\varphi(x,a_\nu):\nu\in w_{\eta^{\frown}0}\cup w_{\eta^{\frown}1}\}$ is inconsistent for any given $\eta\in\null^{\omega_1>}2$.
  \end{enumerate}
  Set $w_{\langle\rangle}=\varnothing$ and $\nu_{\langle\rangle}=\langle\rangle$. At limit case, $w_\eta=\varnothing$ and $\nu_\eta=\bigcup_{\xi\unlhd\eta}\nu_\xi$.\par
  Assume $w_{\eta\lceil\alpha}$, $\nu_{\eta\lceil\alpha}$ is chosen for all $\alpha\leq\text{len}(\eta)$. We choose $w_{\eta^{\frown}i}$ and $\nu_{\eta^{\frown}i}$ for each $i=0,1$ as follows;\par
  Let $p_\eta=\{\varphi(x,a_\nu):\nu\in\bigcup\{w_{\eta\lceil\alpha}:\alpha\leq\text{len}(\eta)\}\}$. 
  We take the least $m_\eta<\omega$ where $p_\eta\cup\{\varphi(x,a_\nu):\nu\in K_{\nu_\eta,m_\eta,\omega_1}\}$ is inconsistent. This $m_\eta$ always exists, otherwise by strong indiscernibility $\{\varphi(x,a_\nu):\nu\in K_{\xi,m,\omega_1}\}$ is consistent for any $\xi\in\null^{\omega_1>}2$ and $m$, which contradicts to the supposition. Note that $m_\eta>0$ because of (a).
  \par
  By minimality of $m_\eta$ and strong indiscernibility, $p_\eta\cup\{\varphi(x,a_\nu):\nu\in K_{\nu_\eta^\frown i,m_\eta-1,\omega_1}\}$ is consistent for $i=0,1$. But by compactness and strong indiscernibility, we have $l_\eta<\omega$ such that $p_\eta\cup\{\varphi(x,a_\nu):\nu\in K_{\nu_\eta,m_\eta,l_\eta}\}$ is inconsistent.\par
  Take $w_{\eta^\frown i}=K_{\nu_\eta^\frown i,m_\eta-1,l_\eta}$ and $\nu_{\eta^\frown i}=\nu_{\eta}^\frown i^\frown0^{m_{\eta}-1\frown}0^{l_{\eta}+1}$ for $i=0,1$.

  Having done the construction, we choose a finite subset $q_\eta\subseteq p_\eta$ for each $\eta$ such that $q_\eta\cup\{\varphi(x,a_\nu):\nu\in w_{\eta^\frown0}\cup w_{\eta^\frown1}\}$ is inconsistent.\par
  Let $\tau_\eta=\{\nu:\varphi(x,a_\nu)\in q_\eta\}$. We may assume $\tau_\eta$ is a finite collection of $K_{\nu,m,\alpha}$s each of which is a finite. 
  Consider $\tau_\eta$ and $w_\eta$ as a tuple, denoting by $\bar{\tau}_\eta$ and $\bar{w}_\eta$. The number of $\sim_{str}$-equivalent classes in $\{\bar{\tau}_\eta:\eta\in\null^{\kappa>}2\}$ and $\{\bar{w}_{\eta^\frown i}:\eta\in\null^{\kappa>}2, i=0,1\}$ are both countable. By Lemma \ref{monochrom lemma}, we have an embedding $h: \null^{\omega>}2\rightarrow \null^{\kappa>}2$ whose range is monochromatic: for any $\eta$ and $\nu$ in $\null^{\omega>}2$, $\bar{\tau}_{h(\eta)}\sim_{str}\bar{\tau}_{h(\nu)}$, and $\bar{w}_{h(\eta)}\sim_{str}\bar{w}_{h(\nu)}$.
  \par
  Define a formula $\psi(x,y)$ and a tree $\la b_\eta:\eta\in\null^{\omega>}2\ra$ such that $\psi(x,b_\eta)=\bigwedge q_{h(\langle\rangle)}\wedge\bigwedge\{\varphi(x,a_\nu):\nu\in w_{h(0^\frown\eta)}\}$. By the minimality of $m_\eta$,  the choice of $w_\eta$, $\nu_\eta$, and the indiscernibility, $\psi(x,y)$ witness SOP$_2$ with $\la b_\eta:\eta\in\null^{\omega>}2\ra$. 
\end{proof}

\begin{prop} \label{weak TP1 prop}
For each $\eta\in\null^{\omega_1>}\omega$, $k\leq\omega$, $m<\omega$, $\alpha\leq\omega_1$, let $K_{\eta, k, m, \alpha}$ to be the set $\{\eta^\frown\nu^\frown0^\beta:\nu\in\null^m k, \beta<\alpha\}$, and $O_\eta$ to be $\{\eta^\frown0^\beta:\beta<\omega_1\}$.\par
  The following are equivalent.
 \begin{enumerate}
     \item $T$ does not have weak-TP$_1$.
     \item For all $\varphi(x,y)$ and strongly indiscernible tree $\la a_\eta:\eta\in\null^{\omega_1>}\omega\ra$, if $\{\varphi(x,a_\nu):\nu\in O_{\langle\rangle}\}$ is consistent, then for each $\eta\in\null^{\omega_1>}\omega$ and $m<\omega$, $\{\varphi(x,a_\nu):\nu\in K_{\eta, \omega,m, \omega_1}\}$ is consistent.
 \end{enumerate}
\end{prop}
  \begin{proof}
  ($1\Leftarrow2$) Use modeling property and compactness.\par
  ($1\Rightarrow2$) Suppose not. Fix $\varphi(x,y)$ and $\la a_\eta:\eta\in\null^{\omega_1>}\omega\ra$. 
  We inductively choose a finite subset $w_\eta\subseteq\null^{\omega_1>}\omega$ and $\nu_\eta\in^{\omega_1>}\omega$ for each $\eta\in\null^{\omega_1>}\omega$ so that the following conditions hold after the construction;\par
  \begin{enumerate}
      \item[(a)] the union of $\{\varphi(x,a_\nu):\nu\in\bigcup\{w_{\eta\lceil\alpha}:\alpha\leq\text{len}(\eta)\}\}$ and $\{\varphi(x,a_\nu):\nu\in O_{\nu_\eta}\}$ is consistent,
      \item[(b)] the union of $\{\varphi(x,a_\nu):\nu\in\bigcup\{w_{\eta\lceil\alpha}:\alpha\leq\text{len}(\eta)\}\}$ and $\{\varphi(x,a_\nu):\nu\in \bigcup_{i<k}w_{\eta^{\frown}i}\}$ is inconsistent.
  \end{enumerate}
\par
  Set $w_{\langle\rangle}=\varnothing$ and $\nu_{\langle\rangle}=\langle\rangle$. At limit case, $w_\eta=\varnothing$ and $\nu_\eta=\bigcup_{\xi\unlhd\eta}\nu_\xi$.\par
  Assume $w_{\eta\lceil\alpha}$, $\nu_{\eta\lceil\alpha}$ is chosen for all $\alpha\leq\text{len}(\eta)$. We choose $w_{\eta^\frown i}$, $\nu_{\eta^\frown i}$ for each $i<\omega$ as follows;\par
  Let $p_\eta=\{\varphi(x,a_\nu):\nu\in\bigcup\{w_{\eta\lceil\alpha}:\alpha\leq\text{len}(\eta)\}\}$. 
  Find the least $m_\eta<\omega$ where $p_\eta\cup\{\varphi(x,a_\nu):\nu\in K_{\nu_\eta,\omega,m_\eta,\omega_1}\}$ is inconsistent. We can always find such $m_\eta$, otherwise by strong indiscernibility $\{\varphi(x,a_\nu):\nu\in K_{\nu_\eta,\omega,m,\omega_1}\}$ is consistent for any $\xi\in\null^{\omega_1>}\omega$ and $m$, which contradicts to the supposition. Note that $m_\eta>0$ because of (a).\par
 By minimiality of $m_\eta$ and strong indiscernibility, $p_\eta\cup\{\varphi(x,a_\nu):\nu\in K_{\nu_\eta^\frown i,\omega,m_\eta-1,\omega_1}\}$ is consistent for any $i<\omega$.\par
  To argue inconsistency, we need an observation on strongly indiscernible trees.
  \begin{obs}
  Let $\la a_\eta:\eta\in\null^{\omega_1>}\omega\ra$ be strongly indiscernible. If $\{\varphi(x,a_\nu):\nu\in K_{\eta, \omega, m, \omega_1}\}$ is inconsistent for some $m<\omega$ and $\eta\in\null^{\omega_1>}\omega$, then there is some $k,l<\omega$ such that $\{\varphi(x,a_\nu):\nu\in K_{\eta, k, m, l}\}$ is inconsistent.
  \end{obs}
  \par
  By the above observation, we have $k_\eta, l_\eta<\omega$ such that $p_\eta\cup\{\varphi(x,a_\nu):\nu\in K_{\eta, k_\eta, m_\eta, l_\eta}\}$ is inconsistent.
  Take $w_{\eta^\frown i}=K_{\eta^\frown i, k_\eta, m_\eta-1, l_\eta}$ 
  and $\nu_{\eta^\frown i}=\nu_{\eta}^\frown i^\frown0^{m_{\eta}-1\frown}0^{l_{\eta}+1}$ for $i<\omega$.
  Note that for any $i_0<\cdots<i_{k_\eta-1}<\omega$, $p_\eta\cup\bigcup_{j<k_\eta}\{\varphi(x,a_\nu):\nu\in w_{\eta^\frown i_j}\}$ is inconsistent by strong indiscernibility.
\par
  Having done the construction, we choose a finite subset $q_\eta\subseteq p_\eta$ for each $\eta$ such that $q_\eta\cup\bigcup_{j<k_\eta}\{\varphi(x,a_\nu):\nu\in w_{\eta^\frown j}\}$ is inconsistent.\par
  Let $\tau_\eta=\{a_\nu:\varphi(x,a_\nu)\in q_\eta\}$. By observation again, we may assume $\tau_\eta$ is a finite collection of $K_{\nu,k,m,l}$s. Considering $\tau_\eta$ as a tuple, the number of $\sim_{str}$-equivalent classes of $\{\bar{\tau}_\eta:\eta\in\null^{\kappa>}\omega\}$ and $\{\bar{w}_{\eta^\frown i}:\eta\in\null^{\kappa>}\omega, i<\omega\}$ are both countable. By Lemma \ref{monochrom lemma 2}, we have an embedding $h: \null^{\omega>}\omega\rightarrow \null^{\omega_1>}\omega$ whose range is monochromatic, that is, for any $\eta$ and $\nu$ in $^{\omega>}\omega$, $\bar{\tau}_{h(\eta)}\sim_{str}\bar{\tau}_{h(\nu)}$, and $\bar{w}_{h(\eta)}\sim_{str}\bar{w}_{h(\nu)}$. \par
  Define a formula $\psi(x,y)$ and a tree $\la b_\eta:\eta\in\null^{\omega>}\omega\ra$ such that $\psi(x,b_\eta)=\bigwedge q_{h(\langle\rangle)}\wedge\bigwedge\{\varphi(x,a_\nu):\nu\in w_{h(0^\frown\eta)}\}$. Then $\psi(x,y)$ and $\la b_\eta:\eta\in\null^{\omega>}\omega\ra$ witness weak TP$_1$.
   \end{proof} 

Finally, we finish proving theorem \ref{nu str}.

\begin{proof}[Proof of Theorem~\ref{nu str}]
(1) Suppose $\varphi$ has TP$_1$. Since any antichain tuple contains $\unlhd$-incomparable pairs, $\varphi$ has $\bar{\nu}^{str}$-TP$_1$ for any antichain  $\bar{\nu}$.\par
To show the converse, note that every finite antichain in $^{\omega>}\omega$ can be strongly embedded into $^n\omega$ for some $n<\omega$. So if $T$ has $\bar
\nu^{str}$-TP$_1$ for some antichain $\bar{\nu}$, then $T$ has weak TP$_1$ by proposition \ref{weak TP1 prop}. By Fact \ref{fact:weak TP1 TP1}, $T$ has TP$_1$.
\par
(2) Note that $T$ has SOP$_2$ if and only if $T$ has TP$_1$. Then by (1), $T$ has $\bar{\nu}^{str}$-TP$_1$. Let $\la a_\eta:\eta\in\null^{\omega>}\omega\ra$ be the witness of $\bar{\nu}^{str}$-TP$_1$-ness. Then $\varphi$ with the subtree $\la a_\eta:\eta\in\null^{\omega>}2\ra$ satisfies $\bar{\nu}^{str}$-SOP$_2$.\par
The converse is clear by proposition \ref{char nsop2}.
\end{proof}

\section{Tree indiscernibility for witnesses of $\sopl $}\label{sec3}

Let us recall the notion of SOP$_1$ in \cite{DS}.

\begin{defn}
 Let $\varphi(x,y)$ be a formula in $T$. We say $\varphi(x,y)$ has {\it 1-strong order property} (SOP$_1$) if there is a tree $\la a_{\eta}\ra_{\eta \in \null^{\omega>}2}$ such that
 \begin{enumerate}
     \item[(i)] For all $\eta \in \null^{\omega}2$, $\{ \varphi (x,a_{\eta \lceil \alpha}) : \alpha < \omega \}$ is consistent, and
     \item[(ii)] For all $\eta, \nu \in \null^{\omega>}2$, $\{ \varphi (x,a_{\eta\coc\llr}), \varphi (x,a_{\eta\coc\lor\coc\nu}) \}$ is inconsistent.
 \end{enumerate}
We say $T$ has SOP$_1$ if it has a SOP$_1$ formula. We say $T$ is NSOP$_1$ if it does not have SOP$_1$.
\end{defn}

For a witness of SOP$_2$, modeling property of strong indiscernibility allows us to assume that the parameter part of the witness is strongly indiscernible. But in the case of SOP$_1$, it can not make the same convenience. If a witness of SOP$_1$ has strong indiscernibility, then it must witness SOP$_2$. Thus strong indiscernibility is not such a suitable tool for dealing with SOP$_1$. So in this section, we develop a tree-indiscernibility which can be applied to witnesses of SOP$_1$.
 
\begin{defn}\label{def of tree-indisc}
For $\bar{\eta}=\langle \eta_0 , ... , \eta_n\rangle$, $\bar{\nu}=\langle \nu_0, ... , \nu_n \rangle$ ($\eta_i , \nu_i \in\tree $ for each $i\leq n$), we say $\bar{\eta}$ and $\bar{\nu}$ are $\alpha$-equivalent ($\bar{\eta}\approx_{\alpha}\bar{\nu}$) if they satisfy
\begin{enumerate}
\item[(i)] $\bar{\eta}$ and $\bar{\nu}$ are $\wedge$-closed,
\item[(ii)] $\eta_i \trianglelefteq \eta_j$ if and only if $\nu_i \tri \nu_j$ for all $i, j \leq n$,
\item[(iii)] $\eta_i{}\coc d \tri \eta_j$ if and only if $\nu_i{}\coc d \tri \nu_j$ for all $i, j \leq n$ and $d <2$.
\end{enumerate}
We say $\bar{\eta}$ and $\bar{\nu}$ are $\beta$-equivalent ($\bar{\eta}\approx_{\beta}\bar{\nu}$) if they satisfy (i), (ii), (iii), and
\begin{enumerate}
\item[(iv)] $\eta_i \coc \llr = \eta_j$ if and only if $\nu_i \coc \llr = \nu_j$ for all $i,j\leq n$.
\end{enumerate}
We say $\bar{\eta}$ and $\bar{\nu}$ are $\gamma$-equivalent ($\bar{\eta}\approx_{\gamma}\bar{\nu}$) if they satisfy (i), (ii), (iii), (iv), and
\begin{enumerate}
\item[(v)] $\eta_{i}\coc\lor=\eta_j$ if and only if $\nu_{i}\coc\lor=\nu_{j}$ for all $i,j\leq n$.
\item[(vi)] $\eta_i=\sigma\coc\lor$ for some $\sigma\in\tree$ if and only if $\nu_i=\tau\coc\lor$ for some $\tau\in\tree$, for all $i\leq n$.
\item[(vii)] $\eta_i=\sigma\coc\llr$ for some $\sigma\in\tree$ if and only if $\nu_i=\tau\coc\llr$ for some $\tau\in\tree$, for all $i\leq n$.
\end{enumerate}

We say $\la a_{\eta}\ra_{\eta\in \tree}$ is $\alpha$-indiscernible ($\beta$, $\gamma$-indiscernible, resp.) if $\bar{\eta}\approx_{\alpha} \bar{\nu}$ ($\bar{\eta}\approx_{\beta} \bar{\nu}$, $\bar{\eta}\approx_{\gamma} \bar{\nu}$, resp.) implies $a_{\bar{\eta}}\equiv a_{\bar{\nu}}$.
\end{defn}

The goal of this section is Theorem \ref{beta}, which says that if a theory has a witness of SOP$_1$, then the theory has a $\beta$-indiscernible witness of SOP$_1$. As we mentioned in the introduction, the outline of proof came from \cite{DS2}. But the proof of \cite{DS2} omits some important step, so that the desired conclusion has not been reached completely. We include a complete proof here, explain what proof of \cite{DS2} omits, and how we complement it. 

To show Theorem \ref{beta} we first prove that if a theory has SOP$_1$, then it has a $\gamma$-indiscernible witness of SOP$_1$ (Lemma \ref{gamma}). And then we obtain a $\beta$-indiscernible witness of SOP$_1$ from this $\gamma$-indiscernible one by taking suitable restriction of the parameter part (Theorem \ref{beta}).

Recall the modeling property of $\alpha$-indiscernibility which appears in \cite{KK}.

\begin{fact} \cite[Proposition~2.3]{KK} \label{a}
For any $\la a_{\eta}\ra_{\eta\in \tree}$, there exists $\la b_{\eta}\ra_{\eta\in \tree}$ such that
\begin{enumerate}
\item[(i)] $\la b_{\eta}\ra_{\eta\in \tree}$ is $\alpha$-indiscernible,
\item[(ii)] for any finite set $\Delta$ of $\mathcal{L}$-formulas and $\wedge$-closed $\bar{\eta}=\la \eta_0, ..., \eta_n \ra$, there exists $\bar{\nu}\approx_{\alpha}\bar{\eta}$ such that $\bar{b}_{\bar{\eta}} \equiv_{\Delta} \bar{a}_{\bar{\nu}}$.
\end{enumerate}
\end{fact}

In order to make proof shorter we introduce some notations.

\begin{nota}
\begin{enumerate}
\item[(i)] For each $\eta\in\tree$, $l(\eta)$ denotes the domain of $\eta$ ({\it{i.e.}}, the length of $\eta$).
\item[(ii)] For each $\eta\in\tree$ with $l(\eta)>0$, $\eta^-$ denotes $\eta_{\res l(\eta)-1}$ and $t(\eta)$ denotes $\eta(l(\eta)-1)$.
\item[(iii)] For $\bar{\eta}=\la \eta_0,...,\eta_n\ra$, $\cl(\bar{\eta})$ denotes $\la \eta_0\wedge \eta_0 , ... , \eta_0\wedge \eta_n\ra \coc ... \coc \la\eta_n\wedge\eta_0,...,\eta_n\wedge\eta_n\ra$.
\item[(iv)] $\eta$ and $\nu$ are said to be incomparable (denoted by $\eta\perp\nu$) if $\eta\not\tri\nu$ and $\nu\not\tri\eta$.
\end{enumerate}
\end{nota}

Note that $\eta=\eta^-\coc t(\eta)$ for all $\eta$ with $l(\eta)>0$. The following remarks will also be useful.

\begin{rmk}\label{b}
Suppose $\la\eta_0,...,\eta_n\ra\approx_{\gamma}\la\nu_0,...,\nu_n\ra$. Then it follows that
\begin{enumerate}
\item[(i)] $\eta_i\wedge\eta_j=\eta_k$ if and only if $\nu_i\wedge\nu_j=\nu_k$,
\item[(ii)] ${\eta_i^-}{\tri\eta_j^-}$ if and only if $\nu_i^-{\tri\nu_j^-}$,
\item[(iii)] if $\eta_i \perp \eta_j$ then $\eta_i^-\wedge\eta_j^-=\eta_i\wedge\eta_j$, 
\item[(iv)] $\eta_i\coc\la d\ra\tri \eta_j^-$ if and only if $\nu_i\coc\la d \ra\tri\nu_j^-$,
\item[(v)] $\eta_i^-\coc\la d\ra\tri \eta_j^-$ if and only if $\nu_i^-\coc\la d \ra\tri\nu_j^-$,
\end{enumerate}
for all $i, j, k\leq n$ and $d\leq 1$.
\end{rmk}

\begin{lemma}\label{gamma}
Suppose $\varphi(x,\bar{y})$ witnesses SOP$_1$. Then there exists a $\gamma$-indiscernible tree $\la d_{\eta}\ra_{\eta\in \tree}$ which witnesses SOP$_1$ with $\varphi$.
\begin{proof}
Suppose $\varphi(x,\bar{y})$ witnesses SOP$_1$ with $\la a_{\eta}\ra_{\eta\in \tree}$. 
For each $\eta\in\tree$, put $b_\eta = a_{\eta\coc \lor}\coc a_{\eta\coc\llr}$. 
By Fact \ref{a}, there exists an $\alpha$-indiscernible $\la c_{\eta} \ra_{\eta\in\tree}$ such that for any $\wedge$-closedv  $\bar{\eta}$ and finite subset $\Delta$ of $\mathcal{L}$-formulas, there exists $\bar{\nu}$ such that $\bar{\nu}\approx_{\alpha}\bar{\eta}$ and $\bar{b}_{\bar{\nu}}\equiv_{\Delta}\bar{c}_{\bar{\eta}}$.
Note that $c_{\eta}$ is of the form $ c_{\eta}^0\coc c_{\eta}^1$ where $|c_{\eta}^0|=|c_{\eta}^1|= |\bar{y}|$ for each $\eta\in\tree$.
 For each $\eta\in\tree$ with $l(\eta)\geq 1$, we define $d'_{\eta}$ by
\[
  d'_{\eta} = \left\{
     \begin{array}{ll}
       c_{\eta^{-}}^0 & \text{ if }\,\,\,t(\eta)=0\\
       c_{\eta^{-}}^1 & \text{ if }\,\,\,t(\eta)=1\\
     \end{array}
   \right.
\]
and put $d_{\eta} = d'_{\lor \coc \eta}$ for each $\eta\in\tree$.
We show that $\varphi$ witnesses SOP$_1$ with $\la d_{\eta}\ra_{\eta\in\tree}$ and $\la d_{\eta}\ra_{\eta\in\tree}$ is $\gamma$-indiscernible.
\begin{claim1}
For all $\eta,\nu\in\tree$, $\{\varphi(x,d_{\eta\coc\llr}),\varphi(x,d_{\eta\coc\lor\coc\nu})\}$ is inconsistent.
\begin{proof}
Define $\psi_0$ and $\psi_1$ by
\[
\begin{array}{ll}
\psi_0(x;\bar{y_0},\bar{y_1})=\varphi(x,\bar{y_0})\wedge \bar{y_1}=\bar{y_1}, \\
\psi_1(x;\bar{y_0},\bar{y_1})=\varphi(x,\bar{y_1})\wedge \bar{y_0}=\bar{y_0}.
\end{array}   
\]

Suppose $\nu=\lr$.
Clearly $\la\lor\coc\eta\ra$ is $\wedge$-closed. So there exists $\mu\in\tree$ such that $\mu\approx_{\alpha}\lor\coc\eta$ and $b_\mu\equiv_\Delta c_{\lor\coc\eta}$, where ${\Delta}{=}{\{\psi_0,\psi_1\}}$.
Since $\{\varphi(x,a_{\mu\coc\llr}),$ $\varphi(x,a_{\mu\coc\lor})\}$ is inconsistent, it follows that $\{\psi_1(x,b_\mu),$ $\psi_0(x,b_\mu)\}$ is inconsistent, $\{\psi_1(x,c_{\lor\coc\eta}),$ 
$\psi_0(x,c_{\lor\coc\eta})\}$ is inconsistent, $\{\varphi(x,c_{\lor\coc \eta}^{1}),$ $\varphi(x,c_{\lor\coc \eta}^{0})\}$ is inconsistent, $\{\varphi(x,d'_{\lor\coc\eta\coc\llr}),$ $\varphi(x,d'_{\lor\coc\eta\coc\lor})\}$ is inconsistent, and hence $\{\varphi(x,d_{\eta\coc\llr}),$ $\varphi(x,d_{\eta\coc \lor\coc\nu})\}$ is inconsistent. 

Now we suppose $\nu\neq\lr$. 
Since $\la\lor\coc\eta,\lor\coc\eta\coc\lor\coc\nu^- \ra$ is $\wedge$-closed, there exists $\mu, \lambda\in\tree$ such that 
\[ 
\begin{array}{r@{\,\,}l}
\la \mu,\lambda \ra & \approx_{\alpha}\la \lor\coc\eta,\lor\coc\eta\coc\lor\coc\nu^- \ra, \\
b_{\mu}b_{\lambda} & \equiv_{\Delta} c_{\lor\coc\eta}c_{\lor\coc\eta\coc\lor\coc\nu^-}.
\end{array}
\]
Then $\mu\coc\lor\tri\lambda\tri\lambda\coc t(\nu)$ by $\approx_{\alpha}$. So $\{\varphi(x,a_{\mu\coc\llr}),\varphi(x,a_{\lambda\coc t(\nu)})\}$ is inconsistent. Thus $\{\psi_1(x,b_\mu),$ $\psi_{t(\nu)}(x,b_{\lambda})\}$ is inconsistent, $\{\psi_1(x,c_{ {\lor}{\coc}{\eta}}),$ $\psi_{t(\nu)}(x,c_{{\lor}{\coc}{\eta}{\coc}{\lor}{\coc}{\nu}^-})\}$ is inconsistent, and hence $\{\varphi(x,d_{\eta\coc\llr}),$ $\varphi(x,d_{\eta\coc\lor\coc\nu})\}$ is inconsistent.
\end{proof}
\end{claim1}
\begin{claim2}
For all $\eta\in\tree$, $\{\varphi(x,d_{\eta\res n}):n\in\omega\}$ is consistent.
\begin{proof}
Since $d_\lr=c^0 _\lr$ and $d_{\eta\res(n+1)}=c^{\eta(n)}_{\lor\coc({\eta\res n})}$ for all $n\in\omega$,
it is enough to show that $\{\psi_0(x,c_{\lr})\}\cup \{\psi_{\eta(n)}(x,c_{\lor\coc(\eta\res n)}):n\in\omega\}$ is consistent. We use compactness. Choose any finite subset of $\{\psi_0(x,c_{\lr})\}$ $\cup$ $ \{\psi_{\eta(n)}(x,c_{\lor\coc(\eta\res n)}):n\in\omega\}$. We may assume that this subset is of the form 
\[
\{\psi_0(x,c_{\lr}),\psi_{\eta(n_1)}(x,c_{\lor\coc(\eta\res n_1)}), ... , \psi_{\eta(n_m)}(x,c_{\lor\coc(\eta\res n_m)})\}
\] 
for some $m\in\omega$ and $n_1 < ... < n_m<\omega$. Clearly $\la\lr,\lor\coc(\eta\res n_1) , ... ,\lor\coc(\eta\res n_m)\rangle$ is $\wedge$-closed. So there exists $\nu_0 , \nu_1 , ..., \nu_m \in\tree$ such that 
\[
\begin{array}{r@{\,\,}l}
\la \nu_0, \nu_1, ..., \nu_m\ra & \approx_{\alpha} \la\lr,\lor\coc(\eta\res n_1),...,\lor\coc(\eta\res n_m) \ra\, \\
b_{\nu_0}b_{\nu_1}...b_{\nu_m} & \equiv_{\Delta} c_{\lr}c_{\lor\coc(\eta\res n_1)} ... c_{\lor\coc(\eta\res n_m)},
\end{array}
\]
where $\Delta=\{\psi_0,\psi_1\}$. Since
\begin{align*}
\lr\coc\lor &\tri \lor\coc(\eta\res{n_1})\\ 
\lor\coc(\eta\res{n_1})\coc\la \eta(n_1)\ra &\tri \lor\coc(\eta\res{n_2}) \\
& \cdots\\
\lor\coc(\eta\res{n_{m-1}})\coc\la\eta(n_{m-1})\ra &\tri \lor\coc(\eta\res{n_m}),
\end{align*}
we have 
\[
\nu_0 \coc\lor\tri\nu_1 \coc\eta(n_1)\tri ...\tri\nu_m \coc \eta(n_m)
\]
by $\approx_{\alpha}$. Thus
\[
\{\varphi(x,a_{\nu_0 \coc \lor}),\varphi(x,a_{\nu_1 \coc \eta(n_1)}),...,\varphi(x,a_{\nu_m \coc\eta(n_m)})\}
\]
is consistent and hence
\[
\{\psi_0 (x,b_{\nu_0}), \psi_{\eta(n_1)}(x,b_{\nu_1}), ..., \psi_{\eta(n_m)}(x,b_{\nu_m})\}
\]
is consistent. Since $b_{\nu_0}b_{\nu_1}...b_{\nu_m} \equiv_{\Delta} c_{\lr}c_{\lor\coc(\eta\res n_1)} ... c_{\lor\coc(\eta\res n_m)}$,
\[
\{\psi_0(x,c_{\lr}),\psi_{\eta(n_1)}(x,c_{\lor\coc(\eta\res n_1)}), ... , \psi_{\eta(n_m)}(x,c_{\lor\coc(\eta\res n_m)})\}
\]
is also consistent. By compactness, $\{\psi_0(x,c_{\lr})\}\cup \{\psi_{\eta(n)}(x,c_{\lor\coc(\eta\res n)}):n\in\omega\}$ is consistent.
\end{proof}
\end{claim2}
So $\varphi$ witnesses SOP$_1$ with $\la d_{\eta}\ra_{\eta\in\tree}$. 

Now we prove that $\la d_{\eta}\ra_{\eta\in\tree}$ is $\gamma$-indiscernible. Suppose that $\la\eta_0, ... ,\eta_n\ra \approx_{\gamma}\la\nu_0,...,\nu_n\ra$. For each $i\leq n$, let $\sigma_i=\lor\coc\eta_i$ and $\tau_i=\lor\coc\nu_i$. It is enough to show that
\[
d'_{\sigma_0}...d'_{\sigma_n}{\equiv} \,\, d'_{\tau_0}...d'_{\tau_n}
\] 
by the choice of $\la d_\eta \ra_{\eta\in\tree}$. First, we show $\la\sigma_0^-,...,\sigma_n^-\ra\approx_\alpha \la\tau_0^-,...,\tau_n^-\ra$. Note that $\la\sigma_0^-,...,\sigma_n^-\ra$ and $\la\tau_0^-,...,\tau_n^-\ra$ may not be $\wedge$-closed. So we prove $\cl(\la\sigma_0^-,...,\sigma_n^-\ra)\approx_{\alpha}\cl(\la\tau_0^-,...,\tau_n^-\ra)$. Note also that $\la\sigma_0,...,\sigma_n\ra$ and $\la\tau_0,...,.\tau_n\ra$ are $\gamma$-equivalent and still $\wedge$-closed. It allows us to use Remark \ref{b}. We will use Remark \ref{b} frequently without much mention.
\begin{claim3}
$\cl(\la\sigma_0^-,...,\sigma_n^-\ra)\approx_{\alpha}\cl(\la\tau_0^-,...,\tau_n^-\ra)$.
\begin{proof}
{\textbf{Subclaim 1.}} $\sigma_i^-\wedge\sigma_j^-\tri\sigma_k^-\wedge\sigma_l^-$ if and only if $\tau_i^-\wedge\tau_j^-\tri\tau_k^-\wedge\tau_l^-$ for all $i,j,k,l\leq n$.
\begin{proof}
Suppose $\sigma_i^-\wedge\sigma_j^-\tri\sigma_k^-\wedge\sigma_l^-$. It is enough to show that $\tau_i^-\wedge\tau_j^-\tri\tau_k^-\wedge\tau_l^-$.
\begin{case1} Then $\tau_i\perp\tau_j$. By Remark \ref{b} we have $\sigma_i\wedge\sigma_j=\sigma_i^-\wedge\sigma_j^-\tri\sigma_k^-\wedge\sigma_l^-\tri\sigma_k,\sigma_l$ and by $\approx_\gamma$, $\tau_i\wedge\tau_j\tri\tau_k,\tau_l$. If $\tau_i\wedge\tau_j=\tau_k$, then $\sigma_i\wedge\sigma_j=\sigma_k$ and hence $\sigma_k^-\wedge\sigma_l^-\tri\sigma_k^-\trn\sigma_i\wedge\sigma_j=\sigma_i^-\wedge\sigma_j^-$, it is a contradiction. Thus $\tau_i\wedge\tau_j\trn\tau_k,\tau_l$ and hence $\tau_i\wedge\tau_j\tri\tau_k^-,\tau_l^-$. So we have $\tau_i^-\wedge\tau_j^-=\tau_i\wedge\tau_j\tri\tau_k^-\wedge\tau_l^-$.
\end{case1}
\noindent {\bf{Case \,2}\, [$\sigma_i\,\not\perp\sigma_j$].}
We may assume $\sigma_i\tri\sigma_j$. Then $\sigma_i^-\tri\sigma_j^-$. Thus $\sigma_i^-=\sigma_i^-\wedge\sigma_j^-\tri\sigma_k^-\wedge\sigma_l^-\tri\sigma_k^-,\sigma_l^-$. By Remark {}\ref{b}(ii) and $\approx_{\gamma}$, we have $\tau_i^-\wedge\tau_j^-=\tau_i^-\tri\tau_k^-,\tau_l^-$. Thus $\tau_i^-\wedge\tau_j^-\tri\tau_k^-\wedge\tau_l^-$. This proves Subclaim 1.
\end{proof}
\begin{subclaim2}
$(\sigma_i^-\wedge\sigma_j^-)\coc\la d\ra \tri\sigma_k^-\wedge\sigma_l^-$ if and only if $(\tau_i^-\wedge\tau_j^-)\coc\la d\ra \tri\tau_k^-\wedge\tau_l^-$, for all $i,j,k,l\leq n$ and $d\leq 1$.
\begin{proof}
Suppose $(\sigma_i^-\wedge\sigma_j^-)\coc\la d \ra\tri\sigma_k^-\wedge\sigma_l^-$. Without loss of generality, we may assume $d=0$ and it suffices to show that $(\tau_i^-\wedge\tau_j^-)\coc\lor\tri\tau_k^-\wedge\tau_l^-$.
\begin{case1}
Then $\sigma_i^-\wedge\sigma_j^-=\sigma_i\wedge\sigma_j$ and $\tau_i^-\wedge\tau_j^-=\tau_i\wedge\tau_j$. Since $\langle\sigma_0,...,\sigma_n\rangle$ is $\wedge$-closed, there exists $m\leq n$ such that $\sigma_i\wedge\sigma_j=\sigma_m$ and by Remark \ref{b}(i), $\tau_i\wedge\tau_j=\tau_m$. Since $\sigma_m\coc\lor=(\sigma_i\wedge\sigma_j)\coc\lor=(\sigma_i^-\wedge\sigma_j^-)\coc\lor\tri\sigma_k^-\wedge\sigma_l^-\tri\sigma_k^-,\sigma_l^-$, it follows that $\tau_m\coc\lor\tri\tau_k^-,\tau_l^-$ by Remark \ref{b}(iv). Since $\tau_i^-\wedge\tau_j^-=\tau_i\wedge\tau_j=\tau_m$, we have $(\tau_i^-\wedge\tau_j^-)\coc\lor\tri\tau_k^-\wedge\tau_l^-$.
\end{case1}
\noindent {\bf{Case \,2}\, [$\sigma_i\,\not\perp\sigma_j$].} We may assume $\sigma_i\tri\sigma_j$. Then $\sigma_i^-\tri\sigma_j^-$ and hence $\sigma_i^-\coc\lor=(\sigma_i^-\wedge\sigma_j^-)\coc\lor\tri\sigma_k^-\wedge\sigma_l^-\tri\sigma_k^-,\sigma_l^-$. Thus $\tau_i^-\coc\lor\tri\tau_k^-,\tau_l^-$ by Remark \ref{b}(v). Since $\tau_i\tri\tau_j$ by $\approx_{\gamma}$, we have $\tau_i^-\tri\tau_j^-$ and hence $(\tau_i^-\wedge\tau_j^-)\coc\lor=\tau_i^-\coc\lor\tri\tau_k^-\wedge\tau_l^-$. This proves Subclaim 2.
\end{proof}
\end{subclaim2} 
Clearly $\cl(\la\sigma_0^-,...,\sigma_n^-\ra)$ and $\cl(\la\tau_0^-,...,\tau_n^-\ra)$ are $\wedge$-closed. Together with Subclaim 1 and 2, it proves Claim 3.
\end{proof}
\end{claim3}
By $\alpha$-indiscernibility of $\la c_\eta\ra_{\eta\in\tree}$, we have $\bar{c}_{\cl(\la\sigma_0^-,...,\sigma_n^-\ra)}\equiv\bar{c}_{\cl(\la\tau_0^-,...,\tau_n^-\ra)}$. In particular, we have $c_{\sigma_0^-}...c_{\sigma_n^-}\equiv c_{\tau_0^-}...c_{\tau_n^-}$. By definition of $\la d'_{\eta}\ra_{\eta\in\tree}$,
\[
d'_{\sigma_0^-\coc\lor}d'_{\sigma_0^-\coc\llr}...d'_{\sigma_n^-\coc\lor}d'_{\sigma_n^-\coc\llr}\,{\equiv}\,\,d'_{\tau_0^-\coc\lor}d'_{\tau_0^-\coc\llr}...d'_{\tau_n^-\coc\lor}d'_{\tau_n^-\coc\llr}.
\]
For each $i\leq n$ and $d\leq 1$, let $\xi_{2i+d}=\sigma_i^- \coc\la d \ra$ and $\zeta_{2i+d}=\tau_i^- \coc\la d \ra$. Then we have
\[
d'_{\xi_0}...d'_{\xi_{2n+1}} \equiv d'_{\zeta_0}...d'_{\zeta_{2n+1}}.
\]
Note that in general, if $m_{\xi_0}...m_{\xi_k}\equiv m_{\zeta_0}...m_{\zeta_k}$ and $i_0 < ... <i_e\leq k$, then  $m_{\xi_{i_0}}...m_{\xi_{i_e}}\equiv m_{\zeta_{i_0}}...m_{\zeta_{i_e}}$.
Since $\la\sigma_0, ... ,\sigma_n\ra \approx_{\gamma}\la\tau_0,...,\tau_n\ra$, we have $2i+t(\sigma_i)=2i+t(\tau_i)$ for all $i\leq n$. Thus 
\[
\begin{array}{rl}
d'_{\sigma_0}...d'_{\sigma_n} 
=d'_{\sigma_0^-\coc  t(\sigma_0)}...d'_{\sigma_n^-\coc  t(\sigma_n)}
=d'_{\xi_{t(\sigma_0)}}...d'_{\xi_{2n+t(\sigma_n)}} 
\equiv & d'_{\zeta_{t(\tau_0)}}...d'_{\zeta_{2n+t(\tau_n)}}\\ 
     = & d'_{\tau_0^-\coc  t(\tau_0)}...d'_{\tau_n^-\coc  t(\tau_n)} \\
     = & d'_{\tau_0}...d'_{\tau_n}
\end{array}
\]
as desired. This shows that $\la d_{\eta}\ra_{\eta\in\tree}$ is $\gamma$-indiscernible, and completes proof of Lemma \ref{gamma}.
\end{proof}
\end{lemma}

Note that even if $i_0 < ... <i_e\leq k$, $j_0 < ... <j_e\leq k$ and $m_{\xi_0}...m_{\xi_k}\equiv m_{\zeta_0}...m_{\zeta_k}$, it is not sure that $m_{\xi_{i_0}}...m_{\xi_{i_e}}\equiv m_{\zeta_{j_0}}...m_{\zeta_{j_e}}$. So if we want to say $d'_{\sigma_0}...d'_{\sigma_n}\equiv d'_{\tau_0}...d'_{\tau_n}$ from 
\[
d'_{\sigma_0^-\coc\lor}d'_{\sigma_0^-\coc\llr}...d'_{\sigma_n^-\coc\lor}d'_{\sigma_n^-\coc\llr}\,{\equiv}\,\,d'_{\tau_0^-\coc\lor}d'_{\tau_0^-\coc\llr}...d'_{\tau_n^-\coc\lor}d'_{\tau_n^-\coc\llr}
\]
in the last paragraph of proof of Lemma \ref{gamma}, it must be guaranteed that $t(\sigma_i)=t(\tau_i)$ for each $i\leq n$. This is why we introduce $\approx_{\gamma}$ and find a $\gamma$-indiscernible witness of SOP$_1$ first, not directly find $\beta$-indiscernible one as in \cite{DS2}. The proof in \cite{DS2} tries to show directly (without using $\gamma$-indiscernibility) the existence of a $\beta$-indiscernible witness of SOP$_1$, so it ends incompletely due to the aforementioned problem.

We will need the following remark when we prove Theorem \ref{beta}.

\begin{rmk}\label{preserving map}
If $h:\tree \to \tree$ satisfies $h(\eta)\coc \la d\ra \tri h(\eta\coc \la d\ra)$ for all $\eta\in\tree$ and $d\leq 1$, then it satisfies
\begin{enumerate}
\item[(i)] $h(\eta)\tri h(\nu)$ if and only if $\eta \tri \nu$,
\item[(ii)] $h(\eta)\wedge h(\nu)=h(\eta \wedge \nu)$
\end{enumerate}
for all $\eta, \nu \in \tree$.
\end{rmk}

\begin{thm}\label{beta}
If $\varphi(x,y)$ witnesses SOP$_1$, then there exists a $\beta$-indiscernible tree $\la e_{\eta}\ra_{\eta\in\tree}$ which witnesses SOP$_1$ with $\varphi$.
\begin{proof}
By Lemma \ref{gamma}, there exists a $\gamma$-indiscernible tree $\la d_\eta\ra_{\eta\in\tree}$ which witnesses SOP$_1$ with $\varphi$. Define a map $h:\tree\to\tree$ by
\[
h(\eta)=\left\{
 \begin{array}{ll}
  \lr                             & \text{ if }\,\,\,\eta=\lr \\
  h(\eta^-)\coc\la 01 \ra & \text{ if }\,\,\,t(\eta)=0\\
  h(\eta^-)\coc\la 1 \ra  & \text{ if }\,\,\,t(\eta)=1,\\
 \end{array}
\right.
\]
and put $e_{\eta}=d_{h(\eta)}$ for each $\eta\in\tree$. 
\begin{claim1}
$\la e_{\eta}\ra_{\eta\in\tree}$ is $\beta$-indiscernible.
\begin{proof}
Let $\la \eta_0,...,\eta_n\ra \approx_{\beta} \la \nu_0,...,\nu_n\ra$. It is enough to show that $\la h(\eta_0),...,h(\eta_n)\ra\approx_{\gamma}\la h(\nu_0),...,h(\nu_n)\ra$. We recall Definition \ref{def of tree-indisc}. Note that $h$ satisfies the assumption of Remark \ref{preserving map}. Thus $h(\bar{\eta})$ and $h(\bar{\nu})$ satisfy (i), (ii) of Definition \ref{def of tree-indisc} by Remark \ref{preserving map}. To show (iii) we suppose $h(\eta_i)\coc \la d\ra \tri h(\eta_j)$ for some $d\leq 1$. Then $h(\eta_i)\trn h(\eta_j)$ hence $\eta_i\trn \eta_j$. If $\eta_i \coc \la (d-1)^2\ra \tri \eta_j$, then we have
\[
h(\eta_i)\coc \la 1-d \ra \tri h(\eta_i \coc \la 1-d\ra)\tri h(\eta_j) ,
\]
a contradiction. Thus $\eta_i\coc \la d\ra \tri \eta_j$. Since $\bar{\eta}\approx_{\beta}\bar{\nu}$, we have $\nu_i\coc \la d\ra \tri \nu_j$. Hence
\[
h(\nu_i)\coc \la d\ra \tri h(\nu_i \coc \la d\ra)\tri h(\nu_j),
\]
and therefore $h(\bar{\eta})$ and $h(\bar{\nu})$ satisfy (iii). Now we suppose $h(\eta_i)\coc\llr=h(\eta_j)$. Since $h$ is an injection, $\eta_i\coc\llr=\eta_j$. Since $\bar{\eta}\approx_{\beta}\bar{\nu}$, we have $\nu_i\coc\llr=\nu_j$. So $h(\nu_i)\coc\llr=h(\nu_j)$. This shows that $h(\bar{\eta})$ and $h(\bar{\nu})$ satisfy (iv).

By definition of $h$, there is no $\eta$ such that $t(h(\eta))=0$. So (v) is vacuously true. By the same reason, $h(\bar{\eta})$, $h(\bar{\nu})$ satisfy (vi) and (vii).
\end{proof}
\end{claim1}
\begin{claim2}
$\varphi$ witnesses SOP$_1$ with $\la e_{\eta}\ra_{\eta\in\tree}$.
\end{claim2}
\begin{proof}
Choose any $\eta\in {^{\omega}}2$. Then $\{ h(\eta\res n):n\in\omega\}$ is linearly ordered by $\tri$ by Remark \ref{preserving map}. Thus $\{\varphi(x,e_{\eta\res n}):n\in\omega\}=\{\varphi(x,d_{h(\eta\res n)}):n\in\omega\}$ is consistent. Choose any $\eta,\nu\in\tree$. Then 
\[
\{\varphi(x,e_{\eta\coc\lor\coc\nu}),\varphi(x,e_{\eta\coc\llr})\}
=\{\varphi(x,d_{h(\eta\coc\lor\coc\nu)}),\varphi(x,d_{h(\eta\coc\llr)})\}
\]
and $h(\eta)\coc\lor\tri h(\eta\coc\lor)\tri h(\eta\coc\lor\coc\nu)$ and $h(\eta\coc\llr)=h(\eta)\coc\llr$. Thus the set $\{\varphi(x,e_{\eta\coc\lor\coc\nu}),\varphi(x,e_{\eta\coc\llr})\}$ is inconsistent.
\end{proof}
This completes proof of Theorem \ref{beta}.
\end{proof}
\end{thm}

So if there exists a witness $\varphi$ of SOP$_1$, then we may choose $\beta$-indiscernible tree which witnesses SOP$_1$ with $\varphi$.

\section{Antichain tree property}\label{secAT}

In this section, we introduce a notion of tree property which is called {\it antichain tree property (ATP)} and explain how to construct an antichain tree in a SOP$_1$-NSOP$_2$ theory. In short, the concept of ATP can be regarded as the dual of the concept of SOP$_2$ in the following sense.

\begin{defn}\label{def of AT}
A tuple $\la \varphi(x,y), \la a_{\eta}\ra_{\eta\in\tree}\ra$ is called an {\it antichain tree} if for all $X\subseteq\tree$, $\{\varphi(x,a_{\eta}):\eta \in X\}$ is consistent if and only if $X$ is an antichain in $\tree$.

We say $\varphi$ has {\it antichain tree property (ATP)} if $\varphi$ forms an antichain tree with some $\la a_{\eta}\ra_{\eta\in\tree}$. We say $T$ has ATP if it has an ATP formula. We say $T$ is {\it NATP} if $T$ does not have ATP.
\end{defn}

Recall the definition of antichain in Section \ref{SOP2}. We say $X\subseteq\tree$ is an antichain if $X$ is pairwisely incomparable ({\it i.e.} $\eta\perp\nu$ for all $\eta,\nu\in X$). Therefore, we can write Definition \ref{def of AT} again as

\begin{defn}
$\la \varphi(x,y), \la a_{\eta}\ra_{\eta\in\tree}\ra$ is called an antichain tree if for all $X\subseteq\tree$, $\{\varphi(x,a_{\eta}):\eta \in X\}$ is consistent if and only if $X$ is pairwisely `incomparable'.

\end{defn} 

And we can write the definition SOP$_2$ again as

\begin{defn}
We say $\la \varphi(x,y), \la a_{\eta}\ra_{\eta\in\tree}\ra$ witnesses SOP$_2$ if for all $X\subseteq\tree$, $\{\varphi(x,a_{\eta}):\eta \in X\}$ is consistent if and only if $X$ is pairwisely `comparable'.
\end{defn}

In this sense, ATP can be thought of as having the dual characteristics to SOP$_2$.

If an antichain tree $\la \varphi(x,y),\paratr \ra$ is given, we can find a witness of SOP$_1$ and a witness of TP$_2$ by restricting the parameter part $\paratr$ as follows.

\begin{prop}\label{AT->SOP1}
If $\la \varphi(x,y), \la a_{\eta}\ra_{\eta\in\tree} \ra$ is an antichain tree, then $\varphi(x,y)$ witnesses SOP$_1$.
\begin{proof}
By compactnss, it is enough to show that for each $n\in\omega$, there exists $h_n:{^{n\geq}}2\to\tree$ such that 
\begin{itemize}
\item[(i)] $\{\varphi(x,b_{\eta\res i}):i\leq n\}$ is consistent for all $\eta\in{^n 2}$,
\item[(ii)] $\{\varphi(x,b_{\eta\coc\lor\coc\nu}),\varphi(x,b_{\eta\coc\llr})\}$ is inconsistent for all $\eta,\nu\in\tren$ with $\eta$ $\coc\lor$ $\coc\nu,$ $\eta$ $\coc\llr$ $\in {^{n\geq}}2$,
\end{itemize}
where $b_\eta=a_{h_n (\eta)}$ for each $\eta\in ^{n\geq}2$. We use induction on $n\in\omega$. Define $h_0:{^{0\geq}}2\to\tree$ by $h_0(\emptyset)=\llr$. For $n\in\omega$, assume such $h_n$ exists. Define $h_{n+1}:{^{n+1\geq}}2\to\tree$ by
\[
h_{n+1}(\eta)=\left\{
 \begin{array}{ll}
  \llr                       & \text{ if }\,\,\,\eta=\lr \\
  \la 011 \ra \coc h_n(\nu)  & \text{ if }\,\,\,\eta=\lor\coc\nu\text{ for some }\nu\in {^{n\geq}}2 \\
  \la 0   \ra \coc h_n(\nu)  & \text{ if }\,\,\,\eta=\llr\coc\nu\text{ for some }\nu\in {^{n\geq}}2.\\
 \end{array}
\right.
\]
For each $\eta\in \null^{n+1\ge}2$, let $b_\eta=a_{h_{n+1} (\eta)}$. Then for each $\eta\in \null^{n+1}2$, $\{h_{n+1}(\eta\res i):i\le n+1\}$ forms an antichain in $\null^{\omega>}2$ by the construction. Thus $\la b_\eta\ra_{\null^{n+1\ge}2}$ satisfies (i) together with $\varphi$. Similarly, we can prove that $\la b_\eta\ra_{\null^{n+1\ge}2}$ satisfies (ii) together with $\varphi$. So we can continue finding $\la h_n\ra_{n<\omega}$ satisfying (i) and (ii) and by compactness, $\varphi$ has SOP$_1$.
\end{proof}
\end{prop}

Let us recall the definition of TP$_2$.

\begin{defn}
We say a formula $\varphi(x,y)$ has {\it tree property of the second kind (TP$_2$)} if there exists an array $\la a_{i,j}\ra_{i,j\in\omega}$ such that
$\{\varphi(x,a_{i,j}):j<\omega\}$ is 2-inconsistent for all $i\in\omega$, and $\{\varphi(x,a_{i,f(i)}):i\in\omega\}$ is consistent for all $f:\omega\to\omega$. We say a theory $T$ has TP$_2$ if there exists a formula having TP$_2$ modulo $T$. We say $T$ is {\it NTP$_2$} if it does not have TP$_2$.
\end{defn}

\begin{prop} \label{at->tp2}
If $\la \varphi(x,y), \la a_{\eta}\ra_{\eta\in\tree} \ra$ is an antichain tree, then $\varphi(x,y)$ witnesses TP$_2$.
\begin{proof}
By compactness, it is enough to show that for any $n\in\omega$, there exists $\la b_{i,j}\ra_{i,j<n}$ such that $\{\varphi(x,b_{i,j}):j<n\}$ is 2-inconsistent for each $i<n$, and $\{\varphi(x,b_{i,f(i)}):i<n\}$ is consistent for all $f:n \to n$. Fix $n\in\omega$. Choose any antichain $\{\eta_0,...,\eta_{n-1}\}$ in $\tree$. Define $h_n:n\times n\to \tree$ by
\[
h_n(i,j)=\eta_i\coc\la 0^j\ra.
\]
Put $b_{i,j}=a_{h_n(i,j)}$ for each $i,j<n$. Then $\{\varphi(x,b_{i,j}):j<n\}$ is 2-inconsistent for all $i<n$ and $\{\varphi(x,b_{i,f(i)}):i<n\}$ is consistent for all $f:n\to n$. 
\end{proof}
\end{prop}

\begin{rmk}
Hence if a theory is ATP, then the theory is TP$_2$ and SOP$_1$. But the converse of each implication is not true in general. Since there are many examples of NSOP$_1$ theory which is TP$_2$ (for example, see \cite[Corollary 3.13, Proposition 3.14]{KR18}), having TP$_2$ does not imply having ATP in general. Similarly, having SOP$_1$ does not imply having ATP in general since there exists an SOP$_1$ theory which is NTP$_2$ (see \cite[Example 1.2]{CKS}). With the already well-known results (see \cite{Con}), we have the following diagram,
\begin{center}
\begin{tikzcd}
 & {\rm NIP}  \arrow{r} & {\rm NTP}_2\arrow{r} & {\rm NATP} \\ 
& {\rm stable} \arrow{r} \arrow{u} & {\rm simple}\arrow{r} \arrow{u} & {\rm NSOP}_1 \arrow{u} 
\end{tikzcd}
\end{center}
and each implication is proper. 
\end{rmk}

From now on we will construct an antichain tree under some classification theoretical hypothesis. 
We begin the construction with the following remark.

\begin{rmk} \label{remark47}
Let $\sigma_0,...,\sigma_n,\sigma_{n+1},\tau_0,...\tau_m,\tau'_0,...,\tau_m'\in\trel$ and assume that there exist $\eta,\nu,\theta_0,...,\theta_m\in\trel$ such that $\tau_i=\eta\coc\theta_i$ and $\tau'_i=\nu\coc\theta_i$ for each $i\leq m$. If
\begin{itemize}
\item[(i)] there exists $l\leq 1$ such that $\sigma_{n+1}\coc\la l\ra\tri\eta,\nu$,
\item[(ii)] for that $l$, no $i\leq n$ satisfies $\sigma_{n+1}\coc\la l\ra\tri\sigma_i$,
\item[(iii)] $\sigma_{n+1}\coc\llr\neq\eta$ and $\sigma_{n+1}\coc\llr\neq\nu$,
\end{itemize}
then ${\rm{cl}}(\la\sigma_0,...,\sigma_n,\sigma_{n+1},\tau_0,...,\tau_m\ra)\approx_{\beta}{\rm{cl}}(\la\sigma_0,...,\sigma_n,\sigma_{n+1},\tau'_0,...,\tau'_m\ra)$.
\end{rmk}

Note that if $T$ has NSOP$_2$ and SOP$_1$, then the hypothesis of Theorem \ref{antichaintree} below always holds.

\begin{thm}\label{antichaintree}
Suppose there exists $\varphi(x,y)$ which witnesses SOP$_1$ and there is no $n\in\omega$ such that $\bigwedge_{i=0}^{n}\varphi(x,y_i)$ witnesses SOP$_2$. Then there exists $\la b_{\eta}\ra_{\eta\in\tree}$ such that $\la\varphi(x,y),\la b_{\eta}\ra_{\eta\in\tree}\ra$ forms an antichain tree.
\begin{proof}
By Theorem \ref{beta} and compactness, there exists an $\beta$-indiscernible $\la a_\eta\ra_{\eta\in\trel}$ which witnesses SOP$_1$ with $\varphi$. Define a map $h:\tree\to\tree$ by
\[
h(\eta)=\left\{
 \begin{array}{ll}
  \llr                             & \text{ if }\,\,\,\eta=\lr \\
  h(\eta^-)^-\coc\la 001 \ra & \text{ if }\,\,\,t(\eta)=0\\
  h(\eta^-)^-\coc\la 011 \ra & \text{ if }\,\,\,t(\eta)=1.\\
 \end{array}
\right.
\]
For each  $i,k\in\omega$ and $\eta,\xi\in\trel$, put
\[
\begin{array}{rl}
L_i=\{h(\nu'):l(\nu')=i\}, & L_i(\eta)=\{\eta\coc\nu:\nu\in L_i\},\\
1_\xi=\{\xi\coc \la 1^d \ra :d\in\omega\}, & 1_\xi(\eta)=\{\eta\coc\nu:\nu\in 1_\xi\},\\
1_{\xi}^{k}=\{\xi\coc \la 1^{0}\ra, ... , \xi\coc \la 1^{k}\ra\}, & 1_\xi^k(\eta)=\{\eta\coc\nu:\nu\in 1_\xi^k\},\\
M_i=L_i\cup 1_{h(\la 0^i\ra)}, & M_i(\eta)=\{\eta\coc\nu:\nu\in M_i\},\\
M_i^k=L_i\cup 1_{h(\la 0^i\ra)}^{k}, & M_i^k(\eta)=\{\eta\coc\nu: \nu\in M_i^k\},\\
m_i^k=h(\la 0^i\ra)\coc \la 1^k\ra, & m_i^k(\eta)=\eta\coc m_i^k,
\end{array}
\]
where $\la 0^0\ra=\la 1^0\ra=\lr$. For each $X\subseteq\trel$, we define $\Phi_X$ by $\{\varphi(x,a_{\eta}):\eta\in X\}$. 

\begin{figure} \label{figure}
\SMALL
\begin{tikzpicture}
[level distance=9mm,
every node/.style={inner sep=0pt},
level 1/.style={sibling distance=34mm},
level 2/.style={sibling distance=50mm},
level 3/.style={sibling distance=18mm},
level 4/.style={sibling distance=25mm},
level 5/.style={sibling distance=13mm},
level 6/.style={sibling distance=17mm}]
\node  
{$\emptyset$} [grow'=up]
child { node {$\lor$} 
        child { node {$\la 0^2 \ra$} 
                child { node {$\la 0^3 \ra$}
                        child { node {$\la 0^4 \ra$}
                                child { node {$\la 0^5 \ra$} }
                                child { node {$h(\la 0^2 \ra)^{* \dagger}$}
                                        child { node {$\la 0^4 10\ra$}}
                                        child { node {$h(\la 0^2 \ra)\coc\llr^\dagger$} 
                                                child { node {$\la 0^4 1^2 0\ra$} } 
                                                child { node {$h(\la 0^2 \ra)\coc\la 1^2 \ra^\dagger$} 
                                                        child { node {$\la 0^4 1^3 0\ra$} }
                                                        child { node {$h(\la 0^2 \ra)\coc\la 1^3 \ra^\dagger$}
                                                                child { node {$\la 0^4 1^4 0\ra$} }
                                                                child { node {$h(\la 0^2 \ra)\coc\la 1^4 \ra$}} } } } } }
                        child { node {$\la 0^3 1 \ra$}
                                child { node {$\la 0^3 10 \ra$} }
                                child { node {$h(\la 01 \ra)^*$} } } }             
                child { node {$h(\lor)$}}}
        child { node{$\la 01 \ra$}
                child { node {$\la 010 \ra$}
                        child { node {$\la 010^2 \ra$}
                                child { node {$\la 010^3 \ra$} }
                                child { node {$h(\la 10 \ra)^*$} }}
                        child { node {$\la 0101 \ra$}
                                child { node {$\la 01010 \ra$} }
                                child { node {$h(\la 1^2 \ra)^*$} } } } 
                child { node {$h(\llr)$} } } } 
child { node {$h(\emptyset)$} } 
;
\end{tikzpicture}
\caption{\Small $m^3_2=h(\la 0^2 \ra)\coc\la 1^3 \ra$ and $M_2^3=L_2\cup 1^3_{h(\la 0^2\ra)}$, where $L_2$ is the set of elements marked *, $1^3_{h(\la 0^2\ra)}$ is the set of elements marked $\dagger$. }
\end{figure}

Let us note some remarks in advance to make the proof easier.
\begin{itemize}
\item[$\bullet$] $L_0=\{\llr\}$.
\item[$\bullet$] $M_0=1_{\llr}\subseteq 1_\lr$.
\item[$\bullet$] $M_0(\eta)=1_\llr(\eta)\subseteq 1_\lr(\eta)=1_\eta$.
\item[$\bullet$] $m_i^k\in M_i^k$ for all $i,k\in\omega$.
\item[$\bullet$] $m_i^k$ is the longest element of $M_i^k$ for all $i,j\in\omega$.
\item[$\bullet$] $t(\eta)=1$ for all $\eta\in M_i^k$ for all $i,k\in\omega$.
\item[$\bullet$] $\bigwedge M_i^k = \lor$ for all $i,k\in\omega$ with $i>0$.
\item[$\bullet$] $\bigwedge M_0^k = \llr$ for all $k\in\omega$.
\end{itemize}

\begin{claimx}
There exists $\eta\in\trel$ such that $\Phi_{M_i (\eta)}$ is consistent for all $i\in\omega$.
\begin{proof}
Suppose not. We construct $\la w_\eta, u_\eta\ra_{\eta\in\trel}$ which satisfies
\begin{enumerate}
\item[(i)] $w_\eta\subseteq \trel$, $u_\eta\in\trel$ for all $\eta\in\trel$,
\item[(ii)] $w_\lr =\emptyset$, $u_\lr =\la 110\ra$,
\item[(iii)] for all $\eta\in\trel$ and $l\leq 1$, there exist $i, k\in\omega$ such that 
\[ w_{\eta\coc\la l\ra}{=}M_i^k(({u_\eta})\coc \la l\ra),\,{\text{\,}}u_{\eta\coc \la l \ra}{=}m_i^k ((u_{\eta})\coc\la l \ra)\coc\la 110\ra,
\]
\item[(iv)] if $l(\eta)$ is a limit, then $u_\eta {=}(\bigcup_{\nu\trn\eta} u_\nu)\coc\la 110 \ra$, $w_\eta =\emptyset$, and $l(u_\eta)$ is limit,
\item[(v)] $\Phi_{p_\eta \cup w_{\eta\coc\lor} \cup w_{\eta\coc\llr}}$ is inconsistent for all $\eta\in\trel$, where $p_\eta =\bigcup_{\nu\tri\eta} w_\nu$,
\item[(vi)] for all $\eta,\nu\in\trel$, if $\eta\tri\nu$ then $u_\eta \trn (u_\nu)^{---}$, 
\item[(vii)] there is no $\xi\in p_\eta$ such that $(u_\eta)^{---}\trn\xi$,
\item[(viii)] for all $\eta,\nu\in\trel$ and $l\leq 1$, if $\eta\coc\lll\tri\nu$ then $u_{\eta}\coc\lll\tri\xi$ for all $\xi\in w_{\nu}$.
\item[(ix)] $\Phi_{p_\eta \cup 1_{(u_\eta)^-}}$ is consistent for all $\eta\in\trel$.
\end{enumerate}
First we put $w_\lr=\emptyset$ and $u_\lr=\la 110\ra$. Then clearly $\Phi_{p_\lr \cup 1_{(u_\lr)^-}}$ is consistent. Suppose that we have constructed $w_\nu$ and $u_\nu$ for all $\nu\tri\eta$. By (ix), $\Phi_{p_\eta \cup 1_{(u_\eta)^-}}$ is consistent so that its subset $\Phi_{p_\eta \cup M_0 ((u_{\eta})^-)}$ is also consistent. 
Since we assume that $\Phi_{M_i ((u_\eta)^-)}$ is not consistent for some $i\in\omega$, there exists $i\in\omega$ such that $\Phi_{p_\eta \cup M_i((u_\eta)^-)}$ is consistent and $\Phi_{p_\eta \cup M_{i+1}((u_\eta)^-)}$ is inconsistent. Note that $t(u_\nu)=0$ for all $\nu\tri\eta$ and $M_{i+1} \subseteq M_i(\la 00\ra)\cup M_i(\la 01\ra)$. So we have 
\[
\begin{array}{rl}
M_{i+1}((u_\eta)^-)\subseteq & M_i((u_\eta)^- \coc\la 00\ra)\cup M_i((u_\eta)^- \coc\la 01\ra) \\
                           = & M_i((u_\eta) \coc\lor)\cup M_i((u_\eta)\coc\llr).
\end{array}
\]
 Thus
\[
\Phi_{p_\eta \cup M_i ((u_{\eta})\coc\lor)\cup M_i((u_{\eta})\coc\llr)}
\]
is inconsistent. By compactness, there exists $k\in\omega$ such that 
\[\Phi_{p_\eta \cup M_i^k ((u_{\eta})\coc\lor)\cup M_i^k((u_{\eta})\coc\llr)}
\]
is inconsistent. Put 
\[
\begin{array}{ll}
w_{\eta\coc\lor}=M_i^k ((u_{\eta})\coc\lor), & w_{\eta\coc\llr}=M_i^k ((u_{\eta})\coc\llr),\\
u_{\eta\coc\lor}=(m_i^k((u_\eta)\coc\lor))\coc\la 110\ra, & u_{\eta\coc\llr}=(m_i^k((u_\eta)\coc\llr))\coc\la 110\ra.
\end{array}
\]
Then $w_{\eta\coc\lor},\, w_{\eta\coc\llr},\, u_{\eta\coc\lor},\, u_{\eta\coc\llr}$ satisfy (i)-(viii). To show they satisfy (ix), note that
\[
\begin{array}{r@{}l}
p_{\eta\coc\la l \ra}\cup 1_{({u_{\eta\coc\la l \ra})}^-} = & {\,\,\,\,} p_{\eta} \cup w_{\eta\coc\la l \ra} \cup 1_{(u_{\eta\coc\la l \ra})^-}\\
= & {\,\,\,\,} p_{\eta} \cup M_i^k((u_\eta)\coc\la l \ra)\cup 1_{(m_i^k((u_{\eta})\coc\la l\ra))\coc\la 11 \ra}\\
\subseteq & {\,\,\,\,} p_{\eta}\cup M_i((u_\eta)\coc\la l \ra).
\end{array}
\]
Thus it is enough to show that $\Phi_{p_{\eta}\cup M_i((u_\eta)\coc\la l \ra)}$ is consistent. Choose any finite subset of $p_\eta \cup M_i((u_\eta)\coc\la l \ra)$. This finite subset can be regarded as a tuple of the form
\[
\la\sigma_0,...,\sigma_n,(u_\eta)\coc\la l\ra\coc\tau_0,...,(u_\eta)\coc\la l\ra\coc\tau_m\ra,
\]
where $\sigma_0,...,\sigma_n\in p_\eta$ and $\tau_0,...,\tau_m\in M_i$. Say $\bar{\sigma}=\la \sigma_0,...,\sigma_n\ra$ and $\bar{\tau}$ $=$ $\la(u_\eta)\coc\la l\ra\coc\tau_0,$ $...,$ $(u_\eta)\coc\la l\ra\coc\tau_m\ra$.
Define a tuple $\bar{\tau}'$ by
\[
\bar{\tau}'=\la(u_\eta)^-\coc\tau_0,...,(u_\eta)^-\coc\tau_m\ra.
\]
Then $\bar{\tau}'\subseteq M_i((u_\eta)^-)$. Recall that $\Phi_{p_\eta \cup M_i((u_\eta)^-)}$ is consistent. Thus $\Phi_{\bar{\sigma}\coc\bar{\tau}'}$ is consistent. So it is enough to show that ${\rm{cl}}(\bar{\sigma}\coc\bar{\tau})\approx_{\beta}{\rm{cl}}(\bar{\sigma}\coc\bar{\tau}')$. By the construction and induction hypothesis, we have
\begin{enumerate}
\item[$\bullet$] $(u_\eta)^{---}\coc\la 11\ra\tri u_\eta^{-}, (u_\eta)\coc\la l \ra$,
\item[$\bullet$] there is no $i\leq n$ such that $(u_\eta)^{---}\trn\sigma_i$.
\end{enumerate}
Thus ${\rm{cl}}(\bar{\sigma}\coc (u_\eta)^{---}\coc\bar{\tau})\approx_{\beta}{\rm{cl}}(\bar{\sigma}\coc (u_\eta)^{---}\coc\bar{\tau}')$ by Lemma \ref{remark47}. In particular we have ${\rm{cl}}(\bar{\sigma}\coc\bar{\tau})\approx_{\beta}{\rm{cl}}(\bar{\sigma}\coc\bar{\tau}')$ as desired. So $w_{\eta\coc\la l\ra}, u_{\eta\coc\la l\ra}$ satisfies (i)-(ix).

For the limit case, suppose $l(\eta)$ is a limit ordinal, $w_\nu$ and $u_\nu$ are constructed for all $\nu\trn\eta$. Put $u_\eta =(\bigcup_{\nu\trn\eta} u_\nu)\coc \la 110\ra$ and $w_\eta = \emptyset$. Then by compactness and $\beta$-indiscernibility, it is clear that $w_\eta$ and $u_\eta$ satisfy (ix). (i)-(viii) are clear. This completes the choice of $\la w_\eta, u_\eta\ra_{\eta\in\trel}$. Note that
\begin{itemize}
\item[$\bullet$] for all $\eta,\nu\in\trel$ and $l\leq 1$, if $\eta\coc\la l\ra\tri\nu$ then $u_\eta\coc\la l\ra\tri u_\nu$.
\end{itemize}

For each $\eta\in\trel$, there exists a finite subset of $p_\eta$  which is inconsistent with $w_{\eta\coc\lor}\cup w_{\eta\coc\llr}$. Let us call this finite subset $q_\eta$. By the similar argument in Lemma \ref{monochrom lemma}, we can prove the following statement. 
\begin{subclaim1} 
There exists a finite subset $q$ of $\trel$ and $\eta^*\in\trel$ such that for any $\eta\trr\eta^*$, there exists $\eta'\trr\eta$ such that $q_{\eta'}=q$.
\begin{proof}
Suppose not. Then
\begin{enumerate}
\item[($*$)] for any finite $q\subseteq\trel$ and $\eta \in\trel$ there exists $\eta^{q}\trr\eta$ such that $q_{\eta'} \neq q$ for all $\eta'\trr\eta^{q}$.
\end{enumerate}
\begin{subsubclaim}
For all $\eta\in\trel$, there exists $\eta'\trr\eta$ such that $q_{\eta''}\not\subseteq p_\eta$ for all $\eta'' \trr \eta'$.
\begin{proof}
Suppose not. Then 
\begin{enumerate}
\item[($\dagger$)] there exists $\eta\in\trel$ such that for all $\eta'\trr\eta$ there is some $\eta''\trr\eta'$ such that $q_{\eta''}\subseteq p_\eta$.
\end{enumerate}
For that $\eta$, there exists $\eta^\emptyset\trr\eta$ such that  $q_{\eta'}\neq \emptyset$ for all $\eta'\trr\eta^{\emptyset}$ by ($*$). And by the choice of $\eta$ in ($\dagger$), there exists $\eta_1\trr\eta^{\emptyset}$ such that $q_{\eta_1}\subseteq p_\eta$. Similarly we can choose $\eta^{q_{\eta_{1}}}\trr\eta_1$ such that $q_{\eta'}\neq q_{\eta_{1}}$ for all $\eta'\trr\eta^{q_{\eta_{1}}}$ by ($*$). And there exists $\eta_2\trr\eta^{q_{\eta_{1}}}$ such that $q_{\eta_2}\subseteq p_\eta$. By iterating this process, one can choose $\langle \eta_i\rangle_{i\in\omega_1}$ which satisfies $q_{\eta_i}\subseteq p_\eta$ for all $i\in\omega_1$, and $q_i\neq q_j$ for all $i,j\in\omega_1$. But it is not possible since the number of finite subsets of countable set $p_\eta$ is countable. This proves Subsubclaim.
\end{proof} 
\end{subsubclaim}
Hence for all $\eta\in\trel$ there exists $\eta'\trr\eta$ such that $q_{\eta''}\not\subseteq p_\eta$ for all $\eta''\trr\eta'$. Put $\eta_0=\lr$. And recursively choose $\eta_{i+1}\trr\eta_i$ which satisfies $q_{\eta''}\not\subseteq p_{\eta_i}$ for all $\eta''\trr\eta_{i+1}$. Let $\eta_\omega$ be $\bigcup_{i\in\omega} \eta_i$. Then $p_{\eta_\omega}=\bigcup_{i\in\omega}p_{\eta_i}$. Since $q_{\eta_\omega}$ is finite, there exists $i\in\omega$ such that $q_{\eta_\omega}\subseteq p_{\eta_i}$. Since $\eta_{\omega}\trr\eta_{i+1}$, we have $q_{\eta_{\omega}}\not\subseteq p_{\eta_i}$, a contradiction. This proves Subclaim 1.
\end{proof}
\end{subclaim1} 
Now we are ready to prove the claim. To show the claim, we construct a witness of SOP$_2$. Fix a finite subset $q$ of $\trel$ and $\eta^*\in\trel$ satisfying Subclaim 1. By the choice of $q$, $\eta^*$, and cofinality of $\omega_1$ we can find a map $f:\trel\to\trel$ such that
\begin{enumerate}
\item[$\bullet$] $f(\emptyset)=\eta^*$,
\item[$\bullet$] $f(\eta)\tri f(\nu)$ for all $\eta,\nu\in\trel$ with $\eta\tri\nu$,
\item[$\bullet$] $f(\eta)\coc\la l\ra\tri f(\eta\coc\la l\ra)$ for all $\eta\in\trel$, $l\leq 1$,
\item[$\bullet$] $f(\eta)\trr\bigcup_{\nu\tri\eta}f(\nu)$ for all $\eta$ such that $l(\eta)$ is a limit ordinal,
\item[$\bullet$] $q_{f(\eta)}=q$ for all $\eta\in\trel$.
\end{enumerate}
Note that for each $\eta\in\trel$, there exists $i,k\in\omega$ such that $w_{f(\eta)\coc\lll}$ is of the form $M_i^k((u_{f(\eta)})\coc\la l \ra)$ for all $l\leq 1$. By using Lemma \ref{monochrom lemma}, we can find a map $g:\tree\to\trel$ and $i,k\in\omega$ such that
\begin{enumerate}
\item[$\bullet$] $g(\eta)\tri g(\nu)$ for all $\eta,\nu\in\tree$ with $\eta\tri\nu$,
\item[$\bullet$] $g(\eta)\coc\lll\tri g(\eta\coc\lll)$ for all $\eta\in\tree$ and $l\leq 1$,
\item[$\bullet$] $w_{f(g(\eta))\coc\lll}=M_i^k((u_{f(g(\eta))})\coc\lll)$ for for all $\eta\in\tree$ and $l\leq 1$.
\end{enumerate}
Let $j=|q|$ and $e=j+|M_i^k|$ (so $e=j+k+2^i$). Say $q=\{\mu_0,...,\mu_{j-1}\}$ and $M_i^k=\{\mu_j,...,\mu_{e-1}\}$. Then $M_i^k(\lambda)=\{\lambda\coc\mu_j,...,\lambda\coc\mu_{e-1}\}$ for all $\lambda\in\trel$. Put $\psi(x,y_0,...,y_{e-1})=\varphi(x,y_0)\wedge ... \wedge\varphi(x,y_{e-1})$ and denote it by $\psi(x,\bar{y})$. For each $\eta\in\tree$ with $l(\eta)>0$, let $\lambda_\eta=(u_{f(g(\eta^-))})\coc\la t(\eta)\ra$ and
\[
\bar{a}_{\eta}=\la a_{\mu_0},...,a_{\mu_{j-1}},a_{(\lambda_\eta) \coc \mu_j},...,a_{(\lambda_\eta)\coc \mu_{e-1}}\ra.
\]   
Note that
\begin{itemize}
\item[$\bullet$] $\{\mu_0,...,\mu_{j-1},(\lambda_\eta)\coc\mu_j,...,(\lambda_\eta)\coc\mu_{e-1}\}=q\cup w_{f(g(\eta^-))\coc\langle t(\eta)\rangle}$.
\end{itemize}
\begin{subclaim2}
$\{\psi(x,\bar{a}_{\eta}),\psi(x,\bar{a}_{\nu})\}$ is inconsistent for any $\eta,\nu\in\tree$ with $\eta\perp\nu$ and $\lor\tri\eta,\nu$. $\{\psi(x,\bar{a}_{\eta\res n}): 0<n\in\omega\}$ is consistent for all $\eta\in\tree$ with $\lor\tri\eta$.
\end{subclaim2}
\begin{proof}
Suppose $\eta\perp\nu$ and $\lor\tri\eta,\nu$. Say $\xi=\eta\wedge\nu$. Clearly $\lor\tri\xi$ and we may assume $\xi\coc\lor\tri\eta$ and $\xi\coc\llr\tri\nu$. Then $f(g(\xi))\coc\lor\tri f(g(\eta^-))\coc\la t(\eta)\ra$ and $f(g(\xi))\coc\llr\tri f(g(\eta^-))\coc\la t(\mu)\ra$ so that $u_{f(g(\xi))}\coc\lor\tri u_{f(g(\xi))\coc\lor}\tri\lambda_\eta$ and $u_{f(g(\xi))}\coc\llr\tri u_{f(g(\xi))\coc\llr}\tri\lambda_\nu$.
Since $\nu_s\in q=q_{f(g(\xi))}\subseteq p_{f(g(\xi))}$ for each $s<j$, we have $(u_{f(g(\xi))})\not\tri\mu_s$ in particular $(u_{f(g(\xi))})\coc\lor\not\tri\mu_s$ for all $s<j$. It is clear that $u_{f(g(\xi))}\coc\lor\tri\lambda_{\xi\coc\lor}$. So we can apply Remark \ref{remark47} to obtain
\[
{\rm{cl}}(\bar{a}_\eta\coc\bar{a}_\nu)\approx_\beta {\rm{cl}}(\bar{a}_{\xi\coc\lor}\coc\bar{a}_\nu).
\]
As we remarked, $\bigwedge M_i^k=\la l\ra$ for some $l\leq 1$. Thus $\mu_s=\la l\ra\coc\mu_s'$ for some $\mu_s'\in\trel$ for each $j\leq s<e$. Thus
\[
\begin{array}{rl}
(\lambda_{\xi\coc\llr})\coc\mu_s=(\lambda_{\xi\coc\llr})\coc\la l\ra\coc\mu_s'=&(u_{f(g(\xi))})\coc\llr\coc\la l\ra\coc\mu_s', \\
(\lambda_{\nu})\coc\mu_s=&(\lambda_{\nu})\coc\la l\ra\coc\mu_s'
\end{array}
\]
for all $j\leq s<e$.
Clearly
\[
\begin{array}{l}
(u_{f(g(\xi))})\coc\llr\not\tri\lambda_{\xi\coc\lor}, \\
(u_{f(g(\xi))})\coc\llr\not\tri\mu_s, \\
(u_{f(g(\xi))})\coc\llr\neq (u_{f(g(\xi))})\coc\llr\coc\la l\ra,\\
(u_{f(g(\xi))})\coc\llr\neq (\lambda_{\nu})\coc\la l\ra
\end{array}
\]
for all $s<j$. By Remark \ref{remark47} again, we have
\[
{\rm{cl}}(\bar{a}_{\xi\coc\lor}\coc\bar{a}_\nu)\approx_\beta {\rm{cl}}(\bar{a}_{\xi\coc\lor}\coc\bar{a}_{\xi\coc\llr}).
\]
As we remarked above
\[
\begin{array}{l}
\{\mu_0,...,\mu_{j-1},(\lambda_{\xi\coc\lor})\coc\mu_j,...,(\lambda_{\xi\coc\lor})\coc\mu_{e-1}\}=q\cup w_{f(g(\xi))\coc\lor},\\
\{\mu_0,...,\mu_{j-1},(\lambda_{\xi\coc\llr})\coc\mu_j,...,(\lambda_{\xi\coc\llr})\coc\mu_{e-1}\}=q\cup w_{f(g(\xi))\coc\llr}.
\end{array}
\]
Hence $\{\psi(x,\bar{a}_{\xi\coc\lor}),\psi(x,\bar{a}_{\xi\coc\llr})\}$ is inconsistent as $\Phi_{q\cup w_{f(g(\xi))\coc\lor}\cup w_{f(g(\xi))\coc\llr}}$ is inconsistent. By $\beta$-indiscernibility $\{\psi(x,\bar{a}_{\eta}),\psi(x,\bar{a}_{\nu})\}$ is also inconsistent.

Suppose $\eta\in {^\omega}2$. Consistency of $\{\psi(x,\bar{a}_{\eta\res n}): 0<n\in\omega\}$ follows the consistency of $\Phi_{p_{\eta}}$. This proves Subclaim 2.
\end{proof}
So $\psi$ witnesses SOP$_2$. But we assume that any conjunction of $\varphi$ does not witness SOP$_2$. Thus we have a contradiction. This proves Claim.
\end{proof}
\end{claimx}
Therefore, there exists $\eta\in\trel$ such that $\Phi_{M_i(\eta)}$ is consistent for all $i\in\omega$. By $\beta$-indiscernibility, we may assume $\eta=\lr$. For each $\eta\in\tree$, put $b_\eta= a_{h(\eta)} $. We show that $\la \varphi(x,b_{\eta})\ra_{\eta\in\tree}$ is an antichain tree. Suppose $\eta\trn\nu$. It easily follows that $h(\eta)^-\coc\llr=h(\eta)$ and $h(\eta)^-\coc\lor\tri h(\nu)$ and hence $\{\varphi(x,b_\eta),\varphi(x,b_\nu)\}$ is inconsistent. Thus if $X\subseteq \tree$ is not an antichain, then the set $\{\varphi(x,b_\eta):\eta\in X\}$ is not consistent.
 
 Let $\{\eta_0,...,\eta_n\}\subseteq\tree$ be an antichain. Then there exist $\nu_0,...,\nu_n$ such that $ \eta_i\tri\nu_i$ for all $i\leq n$, and $l(\nu_0)=...=l(\nu_n)$. So $\{\varphi(x,b_{\nu_0}),...,\varphi(x,b_{\nu_n})\}$ is consistent since $h(\nu_0),...,h(\nu_n)\in M_{l(\nu_0)}$. Thus the only remain is to show that $\cl (\la h(\eta_0),...,h(\eta_n)\ra) \approx_{\beta} \cl(\la h(\nu_0),...,h(\nu_n)\ra)$.
 
 First, we may assume that $\eta_0,...,\eta_n$ are ordered by lexicographic order. That is, $(\eta_i\wedge\eta_{i+1})\coc\lor\tri\eta_i$ for all $i<n$. It is easy to show that for any $i<j\leq n$, 
\begin{enumerate}
\item[$\bullet$] $h(\eta_i)\wedge h(\eta_j) = h(\nu_i)\wedge h(\eta_j)= h(\eta_i)\wedge h(\nu_j)= h(\nu_i)\wedge h(\nu_j)$,
\item[$\bullet$] $h(\eta_i)\wedge h(\eta_j)\coc\lor\tri h(\eta_i),h(\nu_i)$,
\item[$\bullet$] $(h(\eta_i)\wedge h(\eta_j))\coc\llr\tri h(\eta_j),h(\nu_j)$,
\item[$\bullet$] $(h(\eta_i)\wedge h(\eta_j))\coc\llr\neq h(\eta_j),h(\nu_j)$.
\end{enumerate} 
Therefore, we can apply Remark \ref{remark47} repeatedly to obtain
\[
\begin{array}{l}
              \,\,\,\,\,\,\,\,\,\,  \cl(\la h(\eta_0),h(\eta_1),h(\eta_2),...,h(\eta_n)\ra) \\
\approx_{\beta} \cl(\la h(\nu_0),h(\eta_1),h(\eta_2),...,h(\eta_n)\ra) \\
\approx_{\beta} \cl(\la h(\nu_0),h(\nu_1),h(\eta_2),...,h(\eta_n)\ra) \\
\cdots\\
\approx_{\beta} \cl(\la h(\nu_0),h(\nu_1),h(\nu_2),...,h(\nu_n)\ra). \\
\end{array}
\]
This completes proof of Theorem \ref{antichaintree}.
\end{proof}
\end{thm}

\begin{cor}\label{SOP1NSOP2->ATP}
If there exists a formula having SOP$_1$ but any conjunction of it does not have SOP$_2$, then the theory has ATP. The witness of ATP can be selected to be strongly indiscernible.
\begin{proof}
If a theory has SOP$_1$ and does not have SOP$_2$, then the theory has a formula which witnesses SOP$_1$ and any finite conjunction of the formula does not witness SOP$_2$. So we can apply Theorem \ref{antichaintree}. The theory has a witness of ATP. Furthermore, we can obtain a strong indiscernible witness of ATP by using compactness and the modeling property.
\end{proof}
\end{cor}

As we observed in the beginning of this section, one can find witnesses of SOP$_1$ and TP$_2$ from an antichain tree $\la \varphi(x,y),\la a_\eta\ra_{\eta\in\tree}\ra$ by restricting the parameter part $\la a_\eta\ra_{\eta\in\tree}$. But we can not use the same method for finding a witness of SOP$_2$. 

\begin{rmk}\label{AT-x-SOP2} The following are true.
\begin{enumerate}
\item[(i)] Suppose $\la \varphi(x,y), \la a_{\eta}\ra_{\eta\in\tree} \ra$ is an antichain tree. Then there is no $h:{^{2\geq}}2\to\tree$ such that $\la \varphi(x,y), \la b_{\eta}\ra_{\eta\in {^{2 ^{\geq}}}2} \ra$ satisfies the conditions of SOP$_2$, where $b_{\eta}=a_{h(\eta)}$ for each $\eta\in {^{2 \geq}}2$.
\item[(ii)] Suppose $\la \varphi(x,y), \la a_{\eta}\ra_{\eta\in\tree} \ra$ witnesses SOP$_2$. Then there is no $h:{^{2\geq}}2\to\tree$ such that $\la \varphi(x,y), \la b_{\eta}\ra_{\eta\in {^{2 ^{\geq}}}2} \ra$ forms an antichain tree with height 2, where $b_{\eta}=a_{h(\eta)}$ for each $\eta\in {^{2 \geq}}2$.
\end{enumerate}
\begin{proof}
{\rm (i)} To get a contradiction, suppose there exists such $h$. Then $h(\la 00 \ra)$, $h(\la 01 \ra)$, $h(\la 10 \ra)$, and $h(\la 11 \ra)$ are pairwisely comparable in $\tree$, so they are linearly ordered by $\tri$. We may assume $h(\la 00 \ra)$ is the smallest. Since $h(\lor)$ and $h(\la 00 \ra)$ are incomparable, $h(\lor)$ and $h(\la 11 \ra)$ are incomparable. Thus $\{\varphi(x,b_{\lor}),\varphi(x,b_{\la 11 \ra})\}$ is consistent. This is a contradiction.\\
{\rm (ii)} To get a contradiction, suppose there exists such $h$. Then $h(\la 00 \ra)$, $h(\la 01 \ra)$, $h(\la 10 \ra)$, and $h(\la 11 \ra)$ are pairwisely comparable in $\tree$, so they are linearly ordered by $\tri$. We may assume $h(\la 00 \ra)$ is the smallest. Since $h(\lor)$ and $h(\la 00 \ra)$ are incomparable, $h(\lor)$ and $h(\la 11 \ra)$ are incomparable. Thus $\{\varphi(x,b_{\lor}),\varphi(x,b_{\la 11 \ra})\}$ is inconsistent. This is a contradiction.
\end{proof}
\end{rmk}

But it does not mean the existence of an antichain tree prevents the theory from having a witness of SOP$_2$. Under some additional conditions, the existence of an antichain tree is still likely to cause SOP$_2$. For example, if $\la\varphi(x,y),\la a_\eta \ra_{\eta\in\tree}\ra$ forms an antichain tree and `$\tri$' is definable by some formula, then $\tri$ could be a witness of SOP$_2$. We will consider this little more in Section \ref{example of at}.

We end this section with the following remarks. 

\begin{rmk}
If the existence of an antichain tree always implies the existence of a witness of SOP$_2$, then $\sopl=\sopz$ by Corollary \ref{SOP1NSOP2->ATP}.
\end{rmk}

\begin{rmk}
If there exists a NSOP$_2$ theory having an antichain tree, then $\sopl \supsetneq \sopz$ by Proposition \ref{AT->SOP1}.
\end{rmk}

\section{1 strong order property with full consistency}

 As we mentioned in the introduction, in \cite{DS} M. D{\v z}amonja and S. Shelah claimed that if there is a theory which is SOP$_1$-NSOP$_2$, then the theory has a witness of SOP$_1$ which satisfies certain condition (see [\ref{DS}, Claim {\color{red}2.15}], Claim \ref{claim} in this paper). Although the claim was not correct, we can obtain some results from it by modifying its condition. In this section, we introduce a notion of {\it 1 strong order property with full consistency} (SOP$^{\rm fc}_1$) and correct Claim \ref{claim} to make sense by using SOP$^{\rm fc}_1$.

\begin{defn} 
We say a tuple $\la\varphi(x,y), \la a_\eta\ra_{\eta\in\tree}\ra$ witnesses {\it 1 strong order property with full consistency (SOP$^{\;\!fc}_1$)} if for all $X\subseteq\tree$, $\{\varphi(x,a_\eta):\eta \in X\}$ is consistent if and only if $X$ does not contain $\eta\coc\llr$ and $\eta\coc\lor\coc\nu$ for some $\eta, \nu\in\tree$.

We say a theory has SOP$^{\rm f}_1$ if it has an SOP$^{\rm fc}_1$ formula. Otherwise, we say $T$ is NSOP$^{\rm fc}_1$. 
\end{defn}

Clearly if a theory has SOP$^{\rm fc}_1$, then it has SOP$_1$. Let us compare SOP$_1$, SOP$_2$ and SOP$^{\rm fc}_1$ briefly.

\begin{rmk}\label{comparing sop1sop2ssop1}
The notions of tree properties can be understood to have the following characteristics.
\begin{enumerate}
\item[(i)] Let $\la\varphi(x,y), \la a_\eta\ra_{\eta\in\tree}\ra$ is a witness of SOP$_1$. Then for any $X\subseteq\tree$,
\begin{enumerate}
\item[a.] if $X$ is linearly ordered by $\tri$, then $\{\varphi(x,a_\eta):\eta \in X\}$ is consistent,
\item[b.] if $\eta\coc\llr, \eta\coc\lor\coc\nu\in X$ for some $\eta,\nu\in\tree$, then $\{\varphi(x,a_\eta):\eta \in X\}$ is inconsistent,
\item[c.] if $X$ is not linearly ordered by $\tri$ and there do not exist $\eta,\nu\in\tree$ such that $\eta\coc\llr,\eta\coc\lor\coc\nu\in X$, then the definition of SOP$_1$ does not decide whether the set $\{\varphi(x,a_\eta):\eta \in X\}$ is consistent or inconsistent.
\end{enumerate}
\item[(ii)] Let $\la\varphi(x,y), \la a_\eta\ra_{\eta\in\tree}\ra$ is a witness of SOP$_2$. Then for any $X\subseteq\tree$,
\begin{enumerate}
\item[a.] if $X$ is linearly ordered by $\tri$, then $\{\varphi(x,a_\eta):\eta \in X\}$ is consistent,
\item[b.] if $\eta\coc\llr, \eta\coc\lor\coc\nu\in X$ for some $\eta,\nu\in\tree$, then $\{\varphi(x,a_\eta):\eta \in X\}$ is inconsistent,
\item[c.] if $X$ is not linearly ordered by $\tri$ and there do not exist $\eta,\nu\in\tree$ such that $\eta\coc\llr,\eta\coc\lor\coc\nu\in X$, then $\{\varphi(x,a_\eta):\eta \in X\}$ is inconsistent.
\end{enumerate}
\item[(iii)] Let $\la\varphi(x,y), \la a_\eta\ra_{\eta\in\tree}\ra$ is a witness of SOP$^{\rm fc}_1$. Then for any $X\subseteq\tree$,
\begin{enumerate}
\item[a.] if $X$ is linearly ordered by $\tri$, then $\{\varphi(x,a_\eta):\eta \in X\}$ is consistent,
\item[b.] if $\eta\coc\llr, \eta\coc\lor\coc\nu\in X$ for some $\eta,\nu\in\tree$, then $\{\varphi(x,a_\eta):\eta \in X\}$ is inconsistent,
\item[c.] if $X$ is not linearly ordered by $\tri$ and there do not exist $\eta,\nu\in\tree$ such that $\eta\coc\llr,\eta\coc\lor\coc\nu\in X$, then $\{\varphi(x,a_\eta):\eta \in X\}$ is consistent.
\end{enumerate}
\end{enumerate}
\end{rmk}

Thus, roughly speaking, notions of SOP$_2$ and SOP$^{\rm fc}_1$ can be considered to be obtained by removing the ambiguity of SOP$_1$, in the opposite way to each other (by `ambiguity' we mean the situation that there is a subset $X$ of $\tree$ such that we cannot know whether $\{\varphi(x,a_\eta):\eta\in X\}$ is consistent or inconsistent by the definition of SOP$_1$ alone).

As in Section \ref{sec3}, $\beta$-indiscernibility also preserves SOP$^{\rm fc}_1$. Namely,

\begin{prop}\label{beta preserves ssop1}
Suppose $\varphi(x,y)$ witnesses SOP$^{\;\!fc}_1$. Then there is a $\beta$-indiscernible tree  which witnesses SOP$^{\;\!fc}_1$ with $\varphi$.
\begin{proof}
We use the same argument which appears in Lemma \ref{gamma} and Theorem \ref{beta}. Suppose there exists $\la a_\eta \ra_{\eta\in\tree}$ which witnesses SOP$^{\rm f}_1$ with $\varphi$ and let $b_\eta=a_{\eta\coc\lor}\coc a_{\eta\coc\llr}$ for each $\eta\in\tree$. By the modeling property of $\alpha$-indiscernibility, there exists an $\alpha$-indiscernible tree $\la c_\eta \ra_{\eta\in\tree}$ such that for any $\bar{\eta}$ and finite subset $\Delta$ of $\mathcal{L}$-formulas, $\bar{\nu}\approx_{\alpha}\bar{\eta}$ and $\bar{b}_{\bar{\nu}}\equiv_{\Delta} \bar{c}_{\bar{\eta}}$ for some $\bar{\nu}$. 
Then $c_\eta$ can be written of the form $c^0_\eta \coc c^1_\eta$ for each $\eta\in\tree$. For each $\eta\in\tree$ with $l(\eta)\geq 1$, we define $d'_{\eta}$ by $c^0_{\eta^{-}}$ if $t(\eta)=0$, $c^1_{\eta^{-}}$ if $t(\eta)=1$. Put $d_\eta=d'_{\lor\coc\eta}$ for each $\eta\in\tree$. Define $\psi_0$ and $\psi_1$ by $\psi_0(x;\bar{y_0},\bar{y_1})=\varphi(x,\bar{y_0})\wedge \bar{y_1}=\bar{y_1}$ and $\psi_1(x;\bar{y_0},\bar{y_1})=\varphi(x,\bar{y_1})\wedge \bar{y_0}=\bar{y_0}$ respectively.

As we observed in Lemma \ref{gamma}, $\la d_{\eta}\ra_{\eta\in\tree}$ is $\gamma$-indiscernible and $\{\varphi(x,d_{\eta\coc\llr}),$ $\varphi(x,d_{\eta\coc\lor\coc\nu})\}$ is inconsistent for all $\eta,\nu\in\tree$.
\begin{claimx}
For all $X\subseteq\tree$, if there are no $\eta,\nu\in\tree$ such that $\eta\coc\llr,\eta\coc\lor\coc\nu\in X$, then $\{\varphi(x,d_\xi):\xi\in X\}$ is consistent.
\begin{proof}
By compactness, we may assume $X$ is finite. Let $X_0=\{\xi\in X: t(\xi)=0\}$, $X_1=\{\xi\in X: t(\xi)=1\}$. And let $\{\sigma_0 \coc\lor ,..., \sigma_n \coc\lor\}$, $\{\tau_0 \coc\llr ,..., \tau_m \coc\llr\}$ denote $X_0$, $X_1$ respectively. Then $\{\varphi(x,d_\xi):\xi\in X\}$ is equal to
\[
\{\psi_0(x,c_{\sigma_0}),...,\psi_0(x,c_{\sigma_n})\}\cup\{\psi_1(x,c_{\tau_0}),...,\psi_1(x,c_{\tau_n})\}.
\]
Since there are no $\eta,\nu\in\tree$ such that $\eta\coc\llr,\eta\coc\lor\coc\nu \in X$, we have
\begin{itemize}
\item[(i)] $\sigma_i\neq \tau_j$ for all $i\leq n$ and $j\leq m$,
\item[(ii)] $\sigma_i \not\trr \tau_j\coc\lor$ for all $i\leq n$ and $j\leq m$,
\item[(iii)] $\tau_i \not\trr\tau_j\coc\lor$ for all $i,j\leq m$ with $i\neq j$.
\end{itemize}
By choice of $\la c_\eta \ra_{\eta\in\tree}$, there exists $\sigma'_0,...,\sigma'_n,\tau'_0,...,\tau'_m\in\tree$ such that
\[
\la \sigma'_0,...,\sigma'_n,\tau'_0,...,\tau'_m \ra \approx_\alpha \la \sigma_0,...,\sigma_n,\tau_0,...,\tau_m \ra
\]
and
\[
b_{\sigma'_0} ... b_{\sigma'_n} b_{\tau'_0} ... b_{\tau'_m} \equiv_{\Delta} c_{\sigma_0} ... c_{\sigma_n} c_{\tau_0} ... c_{\tau_m}
\]
where $\Delta=\{\psi_0,\psi_1\}$. By $\approx_\alpha$, the conditions (i), (ii), (iii) above hold on $\sigma'_0,$ $...,$ $\sigma'_n,$ $\tau'_0,$ $...,$ $\tau'_m$. Let $X'\{\sigma'_0 \coc\lor,...,\sigma'_n \coc\lor,\tau'_0 \coc\llr,...,\tau'_m \coc\llr\}$. By (i), (ii), (iii), there are no $\eta',\nu'\in\tree$ such that $\eta'\coc\llr,\eta'\coc\lor\coc\nu' \in X'$. Thus the set
\[
\{\varphi(x,a_{\xi'}):\xi'\in X'\}
\]
is consistent since $\la a_\eta \ra_{\eta\in\tree}$ witnesses SOP$^{\rm fc}_1$ with $\varphi$. Thus
\[
\{\psi_0(x,b_{\sigma'_0}),...,\psi_0(x,b_{\sigma'_n})\}\cup\{\psi_1(x,b_{\tau'_0}),...,\psi_1(x,b_{\tau'_m})\}
\]
is consistent. Since $b_{\sigma'_0} ... b_{\sigma'_n} b_{\tau'_0} ... b_{\tau'_m}$ and $c_{\sigma_0} ... c_{\sigma_n} c_{\tau_0} ... c_{\tau_m}$ have the same $\Delta$ type, 
\[
\{\psi_0(x,c_{\sigma_0}),...,\psi_0(x,c_{\sigma_n})\}\cup\{\psi_1(x,c_{\tau_0}),...,\psi_1(x,c_{\tau_n})\}.
\]
is consistent as we desire. This proves claim.
\end{proof}
\end{claimx}
Thus we have a $\gamma$-indiscernible tree $\la d_{\eta}\ra_{\eta\in\tree}$ which witnesses SOP$^{\rm fc}_1$ with $\varphi$.

Define a map $h:\tree\to\tree$ by
\[
h(\eta)=\left\{
 \begin{array}{ll}
  \lr                             & \text{ if }\,\,\,\eta=\lr \\
  h(\eta^-)\coc\la 01 \ra & \text{ if }\,\,\,t(\eta)=0\\
  h(\eta^-)\coc\la 1 \ra  & \text{ if }\,\,\,t(\eta)=1,\\
 \end{array}
\right.
\]
and put $e_{\eta}=d_{h(\eta)}$ for each $\eta\in\tree$. As we observed in Theorem \ref{beta}, $\la e_\eta \ra_{\eta\in\tree}$ is $\beta$-indiscernible. It is easy to show that for all $\eta,\nu\in\tree$,
\[
(\eta\wedge\nu)\coc\llr=\nu~\text{~and~}~(\eta\wedge\nu)\coc\lor\tri\eta
\]
if and only if
\[
(h(\eta)\wedge h(\nu))\coc\llr=h(\nu)~\text{~and~}~(h(\eta)\wedge h(\nu))\coc\lor\tri h(\eta).
\]
Therefore $\la e_\eta \ra_{\eta\in\tree}$ witnesses SOP$^{\rm fc}_1$.
\end{proof}
\end{prop}

\begin{prop}\label{AT->fSOP1}
In a theory $T$, if a formula $\varphi$ has ATP, then $\varphi$ has SOP$^{\;\!fc}_1$.
\begin{proof}
Suppose an antichain tree $\la\varphi, \la a_\eta\ra_{\eta\in\tree}\ra$ is given. Define a map $g:\tree\to\tree$ by

\[
g(\eta)=\left\{
 \begin{array}{ll}
  \lor                             & \text{ if }\,\,\,\eta=\lr \\
  g(\eta^-)^-\coc\la 1000 \ra & \text{ if }\,\,\,t(\eta)=0\\
  g(\eta^-)^-\coc\la 10 \ra & \text{ if }\,\,\,t(\eta)=1.\\
 \end{array}
\right.
\]
Let $b_\eta=a_{g(\eta)}$ for each $\eta\in\tree$. Then $\la\varphi,\la b_\eta \ra_{\eta\in\tree}\ra$ is a witness of SOP$^{\rm fc}_1$.
\end{proof}
\end{prop}

Now we are ready to modify the Claim \ref{claim}.

\begin{prop} \label{modifi}
Suppose that $\varphi(x,y)$ satisfies SOP$_1$, but there is no $n\in\omega$ such that the formula $\varphi_n(x,y_0,...,y_{n-1})=\bigwedge_{k<n}\varphi(x,y_k)$ satisfies SOP$_2$. Then there are witnesses $\langle a_\eta:\eta\in\null^{\omega>}2\rangle$ for $\varphi(x,y)$ satisfying SOP$_1$ which in addition satisfy
\begin{enumerate}
\item [(i)] $\la a_\eta\ra_{\eta\in\tree}$ witnesses SOP$^{\;\!fc}_1$ with $\varphi$,
\item [(ii)] $\la a_\eta\ra_{\eta\in\tree}$ is $\beta$-indiscernible.
\end{enumerate}
\begin{proof}
By Theorem \ref{antichaintree}, there exists an antichain tree. By Proposition \ref{AT->fSOP1}, there exists a witness of SOP$^{\rm fc}_1$. And by Proposition \ref{beta preserves ssop1}, there exists a $\beta$-indiscernible witness of SOP$^{\rm fc}_1$.
\end{proof}
\end{prop}

In fact ATP and SOP$^{\rm fc}_1$ are equivalent. The following proposition shows this along with Proposition \ref{AT->fSOP1}.

\begin{prop}\label{AT<-fSOP1}
In a theory $T$, if a formula $\varphi$ has SOP$^{\;\!fc}_1$, then $\varphi$ has ATP.
\begin{proof}
Suppose a witness of SOP$^{\rm fc}_1$ $\la\varphi, \la a_\eta\ra_{\eta\in\tree}\ra$ is given. Define a map $g:\tree\to\tree$ by

\[
g(\eta)=\left\{
 \begin{array}{ll}
  \llr                             & \text{ if }\,\,\,\eta=\lr \\
  g(\eta^-)^-\coc\la 001 \ra & \text{ if }\,\,\,t(\eta)=0\\
  g(\eta^-)^-\coc\la 011 \ra & \text{ if }\,\,\,t(\eta)=1.\\
 \end{array}
\right.
\]
Let $b_\eta=a_{g(\eta)}$ for each $\eta\in\tree$. Then $\la\varphi,\la b_\eta \ra_{\eta\in\tree}\ra$ is an antichain tree.
\end{proof}
\end{prop}

\begin{cor}
In a theory $T$, a formula $\varphi$ has SOP$^{\;\!fc}_1$ if and only if it has ATP.
\end{cor}

\section{An example of antichain tree}\label{example of at}

In section \ref{secAT}, we showed the existence of an antichain tree in SOP$_1$-NSOP$_2$ context. It is natural to ask if an antichain tree exists without classification theoretical hypothesis. We construct a structure of relational language whose theory has a formula $\varphi(x,y)$ which forms an antichain tree and $\bigwedge_{i<n}\varphi(x,y_i)$ do not witness SOP$_2$ for all $n\in\omega$. Note that $\varphi$ also witnesses SOP$_1$ by Proposition \ref{AT->SOP1}. So this theory witnesses SOP$_1$-NSOP$_2$ at the level of formulas.

\begin{defn}
An antichain $X\subseteq \tree$ is called a maximal antichain if there is no antichain $Y\subseteq \tree$ such that $X\subsetneq Y$.
\end{defn}

We begin the construction with language $\mathcal{L}=\{R\}$ where $R$ is a binary relation symbol.
For each $n\in\omega$, let $\alpha_n \in \omega$ be the number of all maximal antichains in $^{n>}2$,
and $\beta_n$ be the set of all maximal antichains in $^{n>}2$. We can choose a bijection from $\alpha_n$ to $\beta_n$ for each $n\in\omega$, say $\mu_n$. For each $n\in\omega$, let $A_n$ and $B_n$ be finite sets such that $|A_n|=\alpha_n$ and $|B_n|=|{^{n>}}2|$. We denote their elements by
\begin{itemize}
\item[] $A_n =\{a^n _{l} : l < \alpha_n \}$, 
\item[] $B_n =\{b^{n} _{\eta} : \eta \in {^{n>}}2\}$.
\end{itemize}
Let $N_n$ be the disjoint union of $A_n$ and $B_n$ for each $n\in\omega$.

For each $n\in \omega$, let ${\mathcal{C}}_n$ be an $\mathcal{L}$-structure such that ${\mathcal{C}}_n=\la C_n; R^{{\mathcal{C}}_n}\ra$, where $R^{{\mathcal{C}}_n}=\{\la a^n _l, b^n _{\eta}\ra\in A_n {\times} B_n : \eta \in \mu_n (l)\}$. For each $n\in\omega$, let $\iota_n$ be a map from $\alpha_n\cup {}^{n>}2$ to $\alpha_{n+1}\cup{} ^{n+1>}2$ which maps $c\mapsto c$ for all $c\in\alpha_n\cup {}^{n>}2$, and define $\iota_n ^* : C_n \to C_{n+1}$ by $a^n _l \mapsto a^{n+1}_{\iota_n (l)}$ and $b^n_\eta \mapsto b^{n+1}_{\iota_n(\eta)}$.
Then $\iota_n ^*$ is an embedding. So we can regard $\mathcal{C}_n$ as a substructure of $\mathcal{C}_{n+1}$ with respect to $\iota^*_n$. Let $\mathcal{C}$ be $\bigcup_{n<\omega}\mathcal{C}_n$, $A$ and $B$ denote $\bigcup_{n<\omega}A_n$ and $\bigcup_{n<\omega}B_n$ respectively.

Now we check some properties of ${\rm Th}(\mathcal{C})$. Recall that $\{\nu\in\tree:\nu\tri\eta\}$ is linearly ordered by $\tri$ for each $\eta$. This fact will be used frequently to prove the following propositions and remarks.

\begin{prop}\label{AT in Th(M)}
$R(x,y)$ forms an antichain tree in $\rm{Th}(\mathcal{C})$. 
\begin{proof}
By compactness, it is enough to show that for any $n\in\omega$, there exists $\la c_{\eta} \ra_{\eta\in\tren}$ such that $\{R(x,c_\eta):\eta\in X\}$ is consistent if and only if $X$ is an antichain in $\tren$. Fix $n\in\omega$ and let $c_\eta=b^n_\eta$ for each $\eta\in\tren$. 

First we show that if $X\subseteq\tren$ is an antichain, then $\{R(x,c_\eta):\eta\in X\}$ is consistent. Since $X$ is an antichain, there exists a maximal antichain $\bar{X}$ in $\tren$ containing $X$. Let $l=\mu_n^{-1}(\bar{X})$. Then $\la a_l^n, b_\eta^n \ra \in R^{\mathcal{C}_n}$ for all $\eta\in\mu_n(l)=\mu_n(\mu_n^{-1}(\bar{X}))=\bar{X}$. Thus $\{R(x,c_\eta):\eta\in X\}$ is consistent. 

For the converse, suppose $X\subseteq\tren$ is not an antichain. Then there exist $\eta,\nu\in X$ such that $\eta\not\perp\nu$. If $\{R(x,c_\eta):\eta\in X\}$ is consistent, then $\{R(x,b_\eta^n), R(x,b_\nu^n)\}$ is consistent. So there exists $l<\alpha_n$ such that $\mathcal{C}\models R(a_l^n,b_\eta^n)\wedge R(a_l^n,b_\nu^n)$. Thus $\eta,\nu\in \mu_n(l)$. There exists an antichain in $\tren$ containing $\eta$ and $\nu$. It is impossible by the choice of $\eta$ and $\nu$.
\end{proof}
\end{prop}

\begin{rmk}\label{properties of M 1} 
There is no $c_0, c_1, c_2, c_3\in C$ such that
\[
\begin{array}{rl}
\mathcal{C}\models & c_0 \neq c_1 \wedge c_2 \neq c_3 \\
                   & \wedge \exists x (R(x,c_0)\wedge R(x,c_1)) \wedge \exists x (R(x,c_2)\wedge R(x,c_3)) \\
                   & \wedge \neg \exists x (R(x,c_0)\wedge R(x,c_2)) \wedge \neg \exists x (R(x,c_0)\wedge R(x,c_3)) \\
                   & \wedge \neg \exists x (R(x,c_1)\wedge R(x,c_2)) \wedge \neg \exists x (R(x,c_1)\wedge R(x,c_3)).
\end{array}
\]
\begin{proof}
Suppose such $c_0,c_1,c_2,c_3$ exist. We may assume they are $b^n_{\eta_0}, b^n_{\eta_1}, b^n_{\eta_2}, b^n_{\eta_3}$ for some $n\in\omega$ and $\eta_0, \eta_1, \eta_2, \eta_3\,{\in}\,\tren$. Since $\{R(x,b^n_{\eta_0}),$ $R(x,b^n_{\eta_1})\}$ and $\{R(x,b^n_{\eta_2}),$ $R(x,b^n_{\eta_3})\}$ are consistent, $b^n_{\eta_0}\perp b^n_{\eta_1}$ and $b^n_{\eta_2}\perp b^n_{\eta_3}$ in $\tren$. Since $\{R(x,b^n_{\eta_0}),$ $R(x,b^n_{\eta_2})\}$, $\{R(x,b^n_{\eta_0}),$ $R(x,b^n_{\eta_3})\}$, $\{R(x,b^n_{\eta_1}),$ $R(x,b^n_{\eta_2})\}$, $\{R(x,b^n_{\eta_1}),$ $R(x,b^n_{\eta_3})\}$ are inconsistent, $b^n_{\eta_0}\not\perp b^n_{\eta_2}$, $b^n_{\eta_0}\not\perp b^n_{\eta_3}$, $b^n_{\eta_1}\not\perp b^n_{\eta_2}$, $b^n_{\eta_1}\not\perp b^n_{\eta_3}$. Thus $b^n_{\eta_2},b^n_{\eta_3}\tri b^n_{\eta_0}\wedge b^n_{\eta_1}$. Since $\{\nu\in\tren : \nu\tri \nu'\wedge\nu''\}$ is linearly ordered by $\tri$ for all $\nu'\neq\nu''$, we have $b^n_{\eta_2}\not\perp b^n_{\eta_3}$. It is a contradiction.
\end{proof}
\end{rmk}

\begin{rmk}\label{properties of M 2} 
 For all $c_0,...,c_{n-1}\in C$, if $\mathcal{C}\models\neg\exists x (R(x,c_0)\wedge \cdots \wedge R(x,c_n))$, then there exist $i<j< n$ such that $\mathcal{C}\models\neg \exists x (R(x,c_i)\wedge R(x,c_j)$.
\begin{proof}
Suppose $c_0,...,c_{n-1}\in C$ and $\mathcal{C}\models \exists x(R(x,c_i)\wedge R(x,c_j)$ for all $i,j<n$. We may assume $c_i=b^k_{\eta_i}$ for some $k\in\omega$ and $\eta_i \in \trek$ for each $i<n$. As we observed in proof of Proposition \ref{AT in Th(M)}, $\la R(x,y),\la b^k_{\eta}\ra_{\eta\in\trek}\ra$ forms an antichain tree with height $k$. Since $\{R(x,b^k_{\eta_i}),R(x,b^k_{\eta_j})\}$ is consistent for each $i.j<n$, $\{\eta_0,...,\eta_{n-1}\}$ is pairwisely incomparable. Thus $\{\eta_0,...,\eta_{n-1}\}$ is an antichain in $\trek$. Hence $\{(R(x,b^k_{\eta_0}),... R(x,b^k_{\eta_{n-1}})\}$ is consistent in $\mathcal{C}_k$, so it is consistent in $\mathcal{C}$. Therefore $\mathcal{C}\models\exists x (R(x,c_0)\wedge \cdots \wedge R(x,c_n))$.
\end{proof}
\end{rmk}

\begin{prop}
$\bigwedge_{i<n}R(x,y_i)$ does not witness SOP$_2$ for all $n\in\omega$.
\begin{proof} 
To get a contradiction we assume there exists $\la \bar{c}_\eta\ra_{\eta\in\tree}$ which witnesses SOP$_2$ with $\bigwedge_{i<n} R(x,y_i)$, where $\bar{c}_\eta=\la c^0_\eta,...,c^{n-1}_\eta\ra$ for $c_\eta^i\in\mathcal{M}$,  $\mathcal{M}$ is a monster model of $\rm{Th}(\mathcal{C})$. We may assume $\la \bar{c}_\eta\ra_{\eta\in\tree}$ is strong indiscernible. Since
\[
\{ R(x,c_{\lor}^0), ..., R(x,c_{\lor}^{n-1}), R(x,c_{\llr}^0), ..., R(x,c_{\llr}^{n-1})\}
\]
is inconsistent, there exist $i,j<n$ such that $\{R(x,c_{\lor}^i), R(x,c_{\llr}^j)\}$ is inconsistent by Remark \ref{properties of M 2}. By the indiscernibility, $\{R(x,c_{\lor}^i),$ $R(x,c_{\la 11 \ra}^j)\}$, $\{R(x,c_{\la 00 \ra}^i),$ $R(x,c_{\llr}^j)\}$, $\{R(x,c_{\la 00\ra}^i),$ $R(x,c_{\la 11\ra}^j)\}$ are also inconsistent. It violates Remark \ref{properties of M 1} since $\{R(x,c_{\lor}^i),$ $R(x,c_{\la 00 \ra}^i)\}$, $\{R(x,c_{\llr}^j),$ $R(x,c_{\la 11\ra}^j)\}$ are consistent.
\end{proof}
\end{prop}

But ${\rm{Th}}(\mathcal{C})$ has a witness of SOP$_2$. Define $\varphi(x,y)$ by
\[
\begin{array}{l}
x\neq y \\
\wedge \neg \exists w (R(w,x) \wedge R(w,y)) \\
\wedge \exists z ( \exists w (R(w,x) \wedge R(w,z))\wedge  \neg\exists w (R(w,y) \wedge R(w,z)))
\end{array}
\]
Then $\varphi$ says ``$y$ is a predecessor of $x$ in the set of parameters" ({\it{i.e.}}, $y \trn x$).
For each $\eta\in\tree$, let $b_\eta=b_\eta^n$ for some $n\in\omega$. By the constructions of $\mathcal{C}$, $b_\eta$ is well-defined. Then $\{\varphi(x,b_\eta),\varphi(x,b_\nu)\}$ is inconsistent whenever $\eta\perp\nu$. By compactness $\{\varphi(x,b_{\eta\res n}):n<\omega\}$ is consistent for each $\eta\in{^\omega}2$. Thus $\la \varphi(x,y),\la b_\eta \ra_{\eta\in\tree}\ra$ witnesses SOP$_2$ modulo ${\rm{Th}}(\mathcal{C})$.

\subsection*{Acknowledgement} The authors thank the anonymous referee for the advice that has significantly improved our paper. We also thank Jan Dobrowolski for his kind comments on the first version of the paper.

\newcommand{\Addresses}{{
\bigskip
  
\footnotesize
JinHoo Ahn \par
\nopagebreak\textsc{School of Computational Sciences, Korea Institute for Advanced Study, 85 Hoegiro, Dongdaemun-gu, 02455 Seoul, South Korea}\par
\nopagebreak\textit{E-mail address}: \texttt{jinhooahn@kias.re.kr}

\bigskip

\footnotesize
Joonhee Kim \par
\nopagebreak\textsc{School of Mathematics, Korea Institute for Advanced Study, 85 Hoegiro, Dongdaemun-gu, 02455 Seoul, South Korea}\par 
\nopagebreak\textit{E-mail address}: \texttt{kimjoonhee@kias.re.kr}

}}

\Addresses

\end{document}